\documentclass[11pt]{article}
\usepackage{xcolor}
\usepackage{graphicx}
\usepackage{amsmath,amsfonts,amssymb,graphics,amsthm}
\usepackage{comment}
\usepackage{tabularx}
\usepackage[protrusion=true,expansion=true]{microtype}
\usepackage{enumerate}
\usepackage{bbm}
\usepackage{mathrsfs}
\usepackage[margin=1in]{geometry}
\usepackage[shortlabels]{enumitem}

\allowdisplaybreaks

\usepackage[colorlinks,
            linkcolor=blue,
            anchorcolor=green,
            citecolor=blue
            ]{hyperref}

\numberwithin{equation}{section}

\newtheorem{theorem}{Theorem}[section]

\newtheorem{lemma}[theorem]{Lemma}
\newtheorem{proposition}[theorem]{Proposition}
 
\newtheorem{conjecture}[theorem]{Conjecture}

\newtheorem{remark}[theorem]{Remark}
\newtheorem{definition}[theorem]{Definition}

\newcommand{\N}{\mathbb{N}}

\newcommand{\1}{\mathbf{1}}

\newcommand{\sep}{\mathrm{sep}}

\newcommand{\QD}{\mathrm{QD}}
\newcommand{\QS}{\mathrm{QS}}

\newcommand{\CR}{\mathrm{CR}}

\DeclareMathOperator{\SLE}{SLE}
\DeclareMathOperator{\CLE}{CLE}

\DeclareMathOperator{\BCLE}{BCLE}

\def\cM{\mathcal{M}}

\def\cC{\mathcal{C}}

\newcommand{\wt}{\widetilde}

\newcommand{\lp}{\mathrm{loop}}

\newcommand{\potts}{\mathrm{Potts}}

\newcommand{\Loop}{\mathrm{Loop}}

\newcommand{\clockwise}{\circlearrowright} % clockwise
\newcommand{\counterclockwise}{\circlearrowleft} % counterclockwise

\def\cM{\mathcal{M}}

\def\cC{\mathcal{C}}

\newcommand\dd{\mathop{}\!\mathrm{d}} % differential operator
\let\originalleft\left
\let\originalright\right
\renewcommand{\left}{\mathopen{}\mathclose\bgroup\originalleft}
\renewcommand{\right}{\aftergroup\egroup\originalright}

\begin{document}

\title{Three-point connectivity constant for $q$-state Potts spin clusters}
\author{Gefei Cai\thanks{Peking University} \qquad Haoyu Liu$^*$ \qquad Baojun Wu$^*$ \qquad Zijie Zhuang\thanks{University of Pennsylvania}}

\maketitle
\begin{abstract}

Recently, Ang--Cai--Sun--Wu (2024) determined the three-point connectivity constant for two-dimensional critical percolation, confirming a prediction of Delfino and Viti (2010). In this paper, we address the analogous problem for planar critical $q$-state Potts spin clusters. We introduce a continuum three-point connectivity constant and compute it explicitly. Under the scaling-limit conjecture for Potts spin clusters, this quantity coincides with the scaling limit of the properly normalized probability that three points lie in the same spin cluster. The resulting formula agrees with the imaginary DOZZ formula up to an explicit $q$-dependent constant with a geometric interpretation. This answers a question from Delfino--Picco--Santachiara--Viti (2013). The proof exploits the coupling between CLE and LQG, together with the BCLE descriptions of $q$-state Potts scaling limits due to Miller--Sheffield--Werner (2017) and K\"ohler-Schindler and Lehmk\"uhler (2025).

\end{abstract}

\section{Introduction}\label{sec:intro}

The $q$-state Potts model is a fundamental lattice model in statistical mechanics, where each spin takes one of $q$ possible values. It was introduced in physics as a natural generalization of the Ising model to describe magnetism, with the case $q=2$ corresponding to the Ising model. Another important case is $q=3$, known as the three-state Potts model. Through its connection with the random-cluster model~\cite{Edwards-Sokal}, this formulation extends to all real $q>0$. In two dimensions, the Potts model exhibits rich behavior: it is widely believed that on the integer lattice, for $0<q\leq 4$, the model undergoes a continuous phase transition, and that at criticality it has a conformally invariant scaling limit. We refer to~\cite{DC-notes, Manolescu-review} for reviews of recent progress on this model.

Initiated by~\cite{Delfino_2013}, the probabilities that multiple points belong to the same spin cluster in the $q$-state Potts model have been studied and conjectured to be solvable within the framework of conformal field theory (CFT). In the CFT approach, multipoint correlation functions can be expressed using the conformal bootstrap formalism~\cite{bpz-conformal-symmetry} in terms of three-point correlation functions, encoded by structure constants, and other data determined by conformal symmetry. In previous work by the first and third authors, together with Ang and Sun~\cite{ACSW24}, the three-point connectivity constant was determined for percolation clusters and Ising spin clusters.\footnote{The result of~\cite{ACSW24} is based on CLE and does not directly apply to the $q$-state Potts model, since the scaling limit of the $q$-state Potts model, except for $q=2$, cannot be described by CLE but rather by a variant thereof.} The explicit formula for the 3-state Potts model is still missing in physics literature. The main focus of this paper is to define the three-point connectivity constant for $q$-state Potts spin clusters in the continuum and to compute it explicitly (Theorem~\ref{thm: main}). Under the scaling-limit conjecture for Potts spin clusters, this quantity coincides with the scaling limit of a properly normalized discrete probability that three points lie in the same spin cluster (Theorem~\ref{thm:connectivity}). The resulting formula agrees with the imaginary DOZZ formula up to an explicit $q$-dependent normalization constant, and is consistent with the numerical simulations in~\cite{Delfino_2013}.

The rest of Section~\ref{sec:intro} is organized as follows. In Section~\ref{sec:main-result}, we define the multipoint connectivity function of the $q$-state Potts model directly in the continuum setting. Assuming convergence of the $q$-state Potts model, this definition coincides with the scaling limit of its discrete counterpart. We then define the three-point connectivity constant and compute its value. Sections~\ref{sec:proof} and~\ref{sec:outlook} are devoted to outlining the proof strategy and providing an outlook.

\subsection{Three-point connectivity constant of the $q$-state Potts model}\label{sec:main-result}

To motivate our continuum definition, we first briefly recall the Edwards--Sokal coupling between the $q$-state Potts model and Fortuin--Kasteleyn (FK) percolation; see Section~\ref{subsec:edwards-sokal} for details. On the integer lattice, we first sample a critical FK percolation with cluster weight $q$. Then we assign an independently sampled spin to each cluster and aggregate clusters with the same spin, yielding the $q$-state Potts model. Equivalently, to obtain clusters with the same spin, we can color each FK cluster independently red with probability $1/q$ and blue otherwise. Then the red clusters correspond to the cluster of one distinguished spin, and the blue clusters correspond to the remaining spins.

Passing to the continuum, FK percolation is conjectured to converge to the conformal loop ensemble (CLE), a family of conformally invariant random collections of loops introduced by~\cite{shef-cle, shef-werner-cle}. Specifically, the cluster boundaries of FK percolation with cluster weight $q\in(1,4)$\footnote{We require $q>1$ since the Potts spin clusters only make sense when $q>1$. We also exclude the case $q=4$ since in this case the continuum analog of Edwards--Sokal coupling using CLE$_4$ does not give the 4-state Potts spin clusters as different CLE$_4$ loops do not touch each other.} are conjectured to converge to a non-simple $\CLE_{16/\kappa}$ with $\kappa = 4\arccos(-\sqrt{q}/2)/\pi \in( 8/3,4)$. Assuming this convergence, one can then perform the same independent coloring directly on the continuum clusters: color each $\CLE_{16/\kappa}$ cluster red with probability $1/q$ and blue otherwise, and take the closure of the union of red (resp.\ blue) clusters to obtain the red (resp.\ blue) spin clusters. The seminal work~\cite{MSW2017} gives a precise description of the interfaces between the red and blue spin clusters, which we review in Sections~\ref{subsec:pre-bcle} and~\ref{subsec:construct-interface}. In particular, the interfaces are ${\rm SLE}_\kappa$-type curves that are simple; moreover, conditioned on the boundary of a red cluster, the boundaries of the outermost blue clusters have the same law as a simple CLE$_\kappa$. Indeed, the description in~\cite{MSW2017} extends to any red probability $r \in (0,1)$ (for a single spin value $r = 1/q$), and the discrete counterpart is known as the fuzzy Potts model. From now on, we adopt the fuzzy Potts notation and use red to denote the distinguished spin value, and blue to denote the aggregate of the remaining spin values.

For a whole-plane continuum fuzzy Potts configuration $\omega$ with parameters $q\in(1,4)$ and $r=1/q$, let $\{\cC_i\}$ be the collection of all red clusters of $\omega$. For a red cluster $\mathcal{C}_i$, its conditional law given its outer boundary is that of a $\CLE_\kappa$ carpet, so one can define the Miller--Schoug measure~\cite{miller2024existence} $\mathcal{M}_i$ supported on $\cC_i$ (which is a conformally covariant natural measure on the $\CLE_\kappa$ cluster and is measurable with respect to $\mathcal{C}_i$). We define the $n$-point Green's function of red clusters as follows.

\begin{definition}\label{def:n-pt-gf}
    For any $n\in\N$, the $n$-point Green's function $G_n^\potts(z_1,\ldots,z_n)$ for the event that $z_1,\ldots,z_n$ are all in the same red cluster is given by
    \begin{equation}\label{eq:def-gf}
        G_n^\potts(z_1,\ldots,z_n)dz_1 \cdots dz_n=\mathbb{E}\left[\sum_{\cC_j}\prod_{i=1}^n \cM_j(dz_i)\right].
    \end{equation}
\end{definition}

By conformal covariance, we have $G_2^\potts(z_1,z_2)=C_2|z_1-z_2|^{-2(2-d)}$ and $G_3^\potts(z_1,z_2,z_3)=C_3|z_1-z_2|^{-(2-d)}|z_1-z_3|^{-(2-d)}|z_2-z_3|^{-(2-d)}$ for some constants $C_2$ and $C_3$, where $d=1+\frac{2}{\kappa}+\frac{3\kappa}{32}$ is the Hausdorff dimension of CLE$_\kappa$ carpet~\cite{ssw-radii,NacuWerner-carpet}. In particular, the three-point connectivity constant
\begin{equation*}
R(q):=\frac{G_3^{\rm Potts}(z_1,z_2,z_3)}{\sqrt{G_2^{\rm Potts}(z_1,z_2)G_2^{\rm Potts}(z_2,z_3)G_2^{\rm Potts}(z_1,z_3)}}=\frac{C_3}{C_2^{3/2}}
\end{equation*}
does not depend on the positions of $(z_i)_{1\le i\le3}$.

The main result of this paper is an explicit expression for $R(q)$ via the imaginary DOZZ formula. For $\beta > 0$, the imaginary DOZZ formula~\cite{zamolodchikov-gmm} is defined in terms of Zamolodchikov's Upsilon function $\Upsilon_{\beta}(z)$ by the following expression:
\begin{equation*}
    \ln \Upsilon_{\beta}(z)=\int_{0}^\infty \left( \Big (\frac{Q}{2}-z \Big )^2 e^{-t} - \frac{(\sinh((\frac{Q}{2}-z )\frac{t}{2}))^2}{\sinh (\frac{\beta t}{2}) \sinh(\frac{2t}{\beta} )} \right) \frac{dt}{t},\quad 0<\mathrm{Re}(z)< Q,\ \mbox{where } Q=\beta+ \frac1\beta
\end{equation*}
and then analytically continued to $\mathbb{C}$. Then the imaginary DOZZ formula is given by
\begin{align*}
C_{\beta}^{\rm ImDOZZ}(\alpha_1,\alpha_2,\alpha_3)= A\Upsilon_\beta\left(2 \beta-\beta^{-1}+\sum_{j=1}^3 \alpha_j\right)\cdot\prod_{i=1}^3\frac{ \Upsilon_\beta\left(\alpha_1+\alpha_2+\alpha_3-2\alpha_i+\beta\right) }{\left[\Upsilon_\beta\left( 2\alpha_i+\beta\right) \Upsilon_\beta\left( 2\alpha_i+2 \beta-\beta^{-1}\right)\right]^{1 / 2}}
\end{align*}
for real numbers $(\alpha_i)_{1\le i\le3}$, where the normalization factor $A$ is chosen so that $C_{\beta}^{\rm ImDOZZ}(\alpha,\alpha,0)=1$.

\begin{theorem}\label{thm: main}
For the continuum fuzzy Potts model with $q \in (1,4)$, $\kappa = 4\arccos(-\sqrt{q}/2)/\pi\in(\frac{8}{3},4)$, and parameter $r=1/q$, we have
\begin{align}\label{eq:fuzzy}
    R(q)=C(q) \cdot {C_{\beta=\frac{2}{\sqrt{\kappa}}}^{\rm ImDOZZ}\left(\frac{1}{4\beta} - \frac{\beta}{2},\frac{1}{4\beta} - \frac{\beta}{2},\frac{1}{4\beta} - \frac{\beta}{2}\right)},\quad \text{where }C(q)=\sqrt{\kappa/2} \cdot\frac{\sin(\kappa\pi/2)}{\sin(4\pi/\kappa)}.
\end{align}
\end{theorem}

We can relate the $n$-point Green's function $G_n^\potts(z_1,\ldots,z_n)$ in~\eqref{eq:def-gf} to the connectivity probability of the discrete fuzzy Potts model through the following conjecture. For $z_1,\ldots,z_n \in \mathbb{C}$, let $(z_i^\delta)_{1\le i\le n}$ be their lattice approximations on $\delta \mathbb{Z}^2$.

\begin{conjecture}\label{conj}
    Let $\omega^\delta$ be the fuzzy Potts configuration on $\delta \mathbb{Z}^2$ with parameters $q \in (1,4)$ and $r=1/q$.
    Define $P_n^\delta(z_1,\ldots,z_n)$ to be the probability that the points $z_i^\delta$ are all in the same red cluster of $\omega^\delta$, and let $\pi_\delta$ be the probability that there is a red nearest-neighbor path connecting the origin and the unit circle. Then the following limit exists:
    \begin{equation}\label{eq:gf-0}
    P_n(z_1,\ldots,z_n) := \lim_{\delta \to 0} \pi_\delta^{-n} P_n^\delta(z_1,\ldots,z_n).
    \end{equation}
    
    Furthermore, let $\{\mathcal{C}_i^\delta\}$ be the collection of red clusters of $\omega^\delta$, and let $\mathcal{M}_i^\delta$ be the counting measure on vertices in $\mathcal{C}_i^\delta$, normalized by $\delta^2\pi_\delta^{-1}$. Then, under the weak Hausdorff topology,
    \[ (\omega^\delta, \{\mathcal{M}_i^\delta\}) \overset{d}{\rightarrow} (\omega, \{C \cdot \mathcal{M}_i\}), \quad \mbox{as } \delta\to 0. \]
    Here, $\omega$ is the continuum fuzzy Potts configuration, $\mathcal{M}_i$ is the Miller--Schoug measure on $\mathcal{C}_i$ (the scaling limit of $\mathcal{C}_i^\delta$), and $C>0$ is a constant. 
\end{conjecture}

\begin{remark}\label{rmk:measure-converge}
    The analog of Conjecture~\ref{conj} for critical Bernoulli percolation on the triangular lattice has been proved in~\cite{CN06, GPS13, Camia24} based on~\cite{smirnov-cardy}. As explained in \cite[Theorem 4.2]{KSL22}, as long as the convergence of FK percolation interfaces to non-simple CLE$_{16/\kappa}$ is known, the convergence of the fuzzy Potts model to its continuum counterpart can be deduced. Assuming this, we also expect the convergence of the connectivity probabilities in~\eqref{eq:gf-0}, and the natural measure can be extracted using percolation arguments.

\end{remark}

Note that~\eqref{eq:def-gf} is consistent with~\eqref{eq:gf-0} under Conjecture~\ref{conj}. Indeed, we may write
\begin{equation}\label{eq:gf-discrete}
\pi_\delta^{-n}P_n^\delta(z_1^\delta,\ldots,z_n^\delta)\prod_{k=1}^n \delta^2dz_k^\delta=\mathbb{E}\left[\sum_{\mathcal{C}_i^\delta}\prod_{k=1}^n\mathcal{M}_i^\delta(dz_k^\delta)\right].
\end{equation}
Then, under Conjecture~\ref{conj}, the left-hand side of~\eqref{eq:gf-discrete} converges to $P_n(z_1,\ldots,z_n)\prod_{k=1}^n dz_k$, whereas the right-hand side of~\eqref{eq:gf-discrete} converges to $C^n \mathbb{E}\left[\sum_{\cC_j}\prod_{k=1}^n \cM_j(dz_k)\right]$. Comparing with~\eqref{eq:def-gf}, the limit $P_n(z_1,\ldots,z_n)$ of $\pi_\delta^{-n}P_n^\delta(z_1^\delta,\ldots,z_n^\delta)$ equals $C^nG_n^\potts(z_1,\ldots,z_n)$. In particular, we have
\begin{equation}\label{eq:change}
R(q):=\frac{G_3^{\rm Potts}(z_1,z_2,z_3)}{\sqrt{G_2^{\rm Potts}(z_1,z_2)G_2^{\rm Potts}(z_2,z_3)G_2^{\rm Potts}(z_1,z_3)}}=\frac{P_3(z_1,z_2,z_3)}{\sqrt{P_2(z_1,z_2)P_2(z_2,z_3)P_2(z_1,z_3)}}.
\end{equation}

Hence, under Conjecture~\ref{conj}, we obtain the following corollary of Theorem~\ref{thm: main}.
\begin{theorem}\label{thm:connectivity}
For $q\in(1,4)$ and $r=1/q$, assume Conjecture~\ref{conj} holds, and $P_n(z_1,\ldots,z_n)$ is defined as in~\eqref{eq:gf-0}. Then 
\begin{equation}\label{eq:thm1.4}
        \frac{P_3(z_1,z_2,z_3)}{\sqrt{P_2(z_1,z_2)P_2(z_2,z_3)P_2(z_1,z_3)}}=C(q) \cdot C_{\beta=\frac{2}{\sqrt{\kappa}}}^{\rm ImDOZZ}\left(\frac{1}{4\beta} - \frac{\beta}{2},\frac{1}{4\beta} - \frac{\beta}{2},\frac{1}{4\beta} - \frac{\beta}{2}\right).
\end{equation}
\end{theorem}

For $q=3$, we have $C(q=3)=\sqrt{\frac{5+\sqrt{5}}{2}} \approx 1.902113$. Theorem~\ref{thm:connectivity} agrees with the numerical simulations in~\cite[Table 2]{Delfino_2013}\footnote{\cite{Delfino_2013} concerns the probability that three points are in the same spin cluster (not restricted to one distinguished spin), so the three-point connectivity constant simulated there corresponds to $\frac{1}{\sqrt{q}} R(q)$ in this paper.}; see Tables~\ref{tab:kappa-C-DOZZ} and~\ref{tab:Rs-vs-DOZZ}.

\begin{table}[h!]
  \centering
  \caption{Values of $\kappa = 4\arccos(-\sqrt{q}/2)/\pi$, $C(q)$, and the imaginary DOZZ constant in~\eqref{eq:thm1.4} for different values of $q$.}
  \label{tab:kappa-C-DOZZ}
  \begin{tabular}{|l||c|c|c|c|c|c|c|}
    \hline
    $q$ & 1.0 & 1.25 & 1.5 & 1.75 & 2.0 & 2.25 & 2.5 \\
    \hline
    $\kappa$ & 2.666667 & 2.755285 & 2.839139 & 2.920214 & 3.0 & 3.079786 & 3.160861 \\
    \hline
    $C(q)$ & 1.0 & 1.100695 & 1.202563 & 1.306731 & 1.414213 & 1.526056 & 1.643484 \\
    \hline 
    $\mathrm{ImDOZZ}$ & 1.0 & 0.997433 & 0.991314 & 0.983085 & 0.973497 & 0.962951 & 0.951647 \\
    \hline
    $q$ & & 2.75 & 3.0 & 3.25 & 3.5 & 3.75 & 4.0 \\
    \hline
    $\kappa$ &  & 3.244715 & 3.333333 & 3.429802 & 3.539893 & 3.678278 & 4.0 \\
    \hline
    $C(q)$ & & 1.768084 & 1.902113 & 2.049192 & 2.216090 & 2.419665 & $2\sqrt{2}$ \\
    \hline
    $\mathrm{ImDOZZ}$ &  & 0.939642 & 0.926870 & 0.913097 & 0.897767 & 0.879331 & 0.840896 \\
    \hline
  \end{tabular}
\end{table}

\begin{table}[h!]
  \centering
  \caption{Comparison of the three-point connectivity constant $\frac{1}{\sqrt{q}} R(q)$ with numerical estimates from~\cite[Table 2]{Delfino_2013}.}
  \label{tab:Rs-vs-DOZZ}
  \begin{tabular}{|l||c|c|c|c|c|c|c|}
    \hline
    $q$ & 1.0 & 1.25 & 1.5 & 1.75 & 2.0 & 2.25 & 2.5 \\
    \hline
    $R^{\rm num}_s$ & 1.0 & 0.9815(5) & 0.973(2) & 0.9720(5) & 0.9735(2) & 0.9800(3) & 0.9896(12) \\
    $\frac{C(q)}{\sqrt{q}}\cdot \mathrm{ImDOZZ}$ & 1.0 & 0.981964 & 0.973360 & 0.971087 & 0.973497 & 0.979678 & 0.989171 \\
    \hline
    $q$ &  & 2.75 & 3.0 & 3.25 & 3.5 & 3.75 & 4.0 \\
    \hline
    $R^{\rm num}_s$ &  & 1.002(2) & 1.0183(5) & 1.0376(20) & 1.061(3) & 1.093(3) & 1.18(1) \\
    $\frac{C(q)}{\sqrt{q}}\cdot \mathrm{ImDOZZ}$ &  & 1.00184 & 1.01788 & 1.03791 & 1.06345 & 1.09873 & 1.18921 \\
    \hline
  \end{tabular}
\end{table}

A key step in proving Theorem~\ref{thm: main} is the calculation of the conformal radii of interfaces in the fuzzy Potts model. We collect these results here, as they may be of independent interest. Let $\mathcal{L}_{R \to B}$ (resp.\ $\mathcal{L}_{B \to R}$) be the boundary of the outermost blue (resp.\ red) cluster surrounding the origin in a continuum fuzzy Potts model in the unit disk $\mathbb{D}$ with red (resp.\ blue) boundary conditions; see Section~\ref{subsec:construct-interface} for a precise definition.

For a simply connected domain $D \subset \mathbb{C}$ and a point $z \in D$, the conformal radius of $D$ viewed from $z$ is defined by $\CR(z,D):=|\psi'(0)|$, where $\psi$ is a conformal map from $\mathbb{D}$ to $D$ such that $\psi(0)=z$.
For a non-self-crossing loop $\eta$ on $\mathbb{C}$ surrounding the origin, let $D(\eta)$ be the connected component of $\mathbb{C} \setminus \eta$ containing the origin.  Write $\mathsf{R}_{R \to B}=\mathrm{CR}(0,D(\mathcal{L}_{R \to B}))$ and $\mathsf{R}_{B \to R}=\mathrm{CR}(0,D(\mathcal{L}_{B \to R}))$.

\begin{proposition}\label{prop:cr-moment-intro}
    Consider the continuum fuzzy Potts model with $q \in (1,4)$, $\kappa = 4\arccos(-\sqrt{q}/2)/\pi \in (8/3,4)$, and $r = 1/q$. 
    
    \begin{itemize}
        \item Let $\lambda>\frac{2}{\kappa}+\frac{3\kappa}{32}-1$ and $\theta=\frac{\pi}{\kappa} \sqrt{(4-\kappa)^2-8\kappa\lambda}$. Then
        \begin{equation}\label{eq:cr-1-23-intro}
            \mathbb{E}[(\mathsf{R}_{R \to B})^\lambda]=\frac{\cos(\frac{\pi(4-\kappa)}{\kappa})}{\cos(\theta)}.
        \end{equation}
        Moreover, if $\lambda \le \frac{2}{\kappa}+\frac{3\kappa}{32}-1$, the left-hand side of~\eqref{eq:cr-1-23-intro} is infinite.

        \item Let $-\lambda_0$ be the unique solution in $(0,1-\frac{\kappa}{8}-\frac{3}{2\kappa})$ to the equation
        \begin{equation*}
            \frac{\sin( \frac{\pi(\kappa-1)}{\kappa} \sqrt{(4-\kappa)^2+8\kappa x})}{\sin( \frac{\pi(2-\kappa)}{2\kappa}\sqrt{(4-\kappa)^2+8\kappa x})}=-2 \cos(\pi \tfrac{4-\kappa}{2}).
        \end{equation*}
        For $\lambda>\lambda_0$ and $\theta=\frac{\pi}{\kappa} \sqrt{(4-\kappa)^2-8\kappa\lambda}$,
        \begin{equation}\label{eq:cr-23-1-intro}
            \mathbb{E}[(\mathsf{R}_{B \to R})^\lambda] = \frac{1}{2 \cos(\frac{\pi(4-\kappa)}{\kappa})} \cdot \frac{\sin((\kappa-2) \theta) + 2 \cos(\pi \frac{4-\kappa}{2}) \sin((2-\frac{\kappa}{2}) \theta)}{\sin((\kappa-1) \theta) + 2 \cos(\pi \frac{4-\kappa}{2}) \sin((1-\frac{\kappa}{2}) \theta)}.
        \end{equation}
        Moreover, if $\lambda \le \lambda_0$, the left-hand side of~\eqref{eq:cr-23-1-intro} is infinite.
    \end{itemize}
    
\end{proposition}

\subsection{Proof strategy}\label{sec:proof}

The proof of Theorem~\ref{thm: main} proceeds in two main steps.

\emph{Step 1: Reduction to a universal constant times the imaginary DOZZ factor.}
Let $\widetilde{\mathrm{SLE}}^{\mathrm{sep}}_\kappa$ be the counting measure on blue-red interfaces (with blue on the outer side and red on the inner, see Section~\ref{subsec:construct-interface})
that separate $0$ and $\infty$ in a whole-plane continuum fuzzy Potts model, and let $\mathrm{SLE}^{\mathrm{sep}}_\kappa$ be the counting measure on loops in a whole-plane CLE$_\kappa$ that separate $0$ and $\infty$. We prove the up-to-constant identification
\begin{equation}\label{eq:equal-sleloop}
\widetilde{\mathrm{SLE}}^{\mathrm{sep}}_\kappa \;=\; \mathsf{C}(\kappa)^{-2} \,\mathrm{SLE}^{\mathrm{sep}}_\kappa,
\qquad 
\mbox{where} \quad \mathsf{C}(\kappa)=\sqrt{\frac{\mathbb{E}[\log \mathsf{R}_{R\to B}]+\mathbb{E}[\log \mathsf{R}_{B\to R}]}{\mathbb{E}[\log \mathsf{R}_{R\to B}]}} ,
\end{equation}
see Theorem~\ref{thm:ratio}. Combining this with the CLE cluster Green's function representation of~\cite{ACSW24} yields that the fuzzy Potts $n$-point Green's function equals $\mathsf{C}(\kappa)^{-2}$ times its CLE counterpart, and hence the three-point connectivity constant equals $\mathsf{C}(\kappa)$ times its CLE counterpart, namely the imaginary DOZZ constant as shown in~\cite{ACSW24}; this proves Theorem~\ref{thm: main} once $\mathsf{C}(\kappa)$ is made explicit.

The identification $\widetilde{\mathrm{SLE}}^{\mathrm{sep}}_\kappa = \mathsf{C}(\kappa)^{-2} \mathrm{SLE}^{\mathrm{sep}}_\kappa$ follows from the strategy in~\cite[Section 6]{ACSW24} by identifying both measures as stationary distributions of a Markov chain on loops. To prove the stationarity of $\widetilde{\mathrm{SLE}}^{\mathrm{sep}}_\kappa$, we also need some conformal welding arguments as input. Normalizing constants are read off from the drift of the logarithm of the conformal radius. 

\emph{Step 2: Explicit evaluation of $\mathsf{C}(\kappa)$.}
We compute the moments of the conformal radii $\mathsf{R}_{R\to B}$ and $\mathsf{R}_{B\to R}$ as stated in Proposition~\ref{prop:cr-moment-intro}, from which we extract $\mathbb{E}[\log \mathsf{R}_{R\to B}]$ and $\mathbb{E}[\log \mathsf{R}_{B\to R}]$ and hence $\mathsf{C}(\kappa)$ (Theorem~\ref{thm:value}). It turns out that $\mathsf{C}(\kappa)$ is equal to $C(q)$ in~\eqref{eq:fuzzy}. The computation of moments relies on the BCLE description of fuzzy Potts interfaces~\cite{MSW2017} and the results of~\cite{LSYZ24}.

\subsection{Outlook}\label{sec:outlook}

We list several future directions.

\begin{itemize}
    \item One natural direction is to compute the multipoint connectivity function of the $q$-state Potts model by rigorously implementing the conformal bootstrap formalism~\cite{bpz-conformal-symmetry}. As shown in Section~\ref{subsec:proof-1.2}, it coincides with the multipoint connectivity of simple CLE$_\kappa$ up to an explicit constant. Nevertheless, this problem remains open even for percolation and Ising clusters~\cite{ACSW24}. See~\cite{NRJ24} for recent progress on the physics side.
    
    \item As shown in~\cite{KW16}, the measure $\mathrm{SLE}^{\mathrm{sep}}_\kappa$ coincides with the SLE loop measure~\cite{zhan-loop-measures} on the whole plane restricted to loops separating 0 and $\infty$. Recently, it has been proved that the SLE loop measure is characterized by the \textit{conformal restriction} property~\cite{baverez2024cftsleloopmeasures, cg25}. One natural question is whether we can directly verify that $\widetilde{\mathrm{SLE}}^{\mathrm{sep}}_\kappa$ satisfies the conformal restriction property, so that we can apply~\cite{baverez2024cftsleloopmeasures, cg25} to prove~\eqref{eq:equal-sleloop}. 

    \item Another natural direction is to extend the result of Theorem~\ref{thm: main} to the fuzzy Potts model for any red probability $r \in (0,1)$. We believe that the method of~\cite{ACSW24} can treat this case, except that two inputs are missing. First, the natural measure on red clusters for general $r \in (0,1)$ has not been defined: conditional on the outer boundary of a red cluster, the interior does not have the law of a CLE$_\kappa$ carpet, so one needs to extend~\cite{miller2024existence} to this setting in order to define the continuum multipoint Green's function. Second, a key input in~\cite{ACSW24} for deriving the three-point connectivity constant is the joint distribution of quantum boundary lengths of CLE$_\kappa$ loops coupled with LQG disks~\cite{MSW22-simple}; we need a version for fuzzy Potts interfaces.
    
    \item The two-dimensional Ising and $3$-state Potts models are of particular interest due to their connections to minimal model CFTs. Minimal models are rational CFTs constructed from a finite set of representations of the Virasoro algebra, and they are classified into three series (A, D, and E)~\cite{bpz-conformal-symmetry,CFT-textbook}. Within this classification, the Ising model is the simplest case in the $A$-series, while the $3$-state Potts model is the simplest example in the $D$-series. In the Ising case, the primary fields of the minimal model CFT are the spin and energy operators; this connection has been rigorously established via discrete complex analysis~\cite{Smi10, Ising-spin, Ising-energy, CHI-ising}. By contrast, the $3$-state Potts model involves a richer collection of primary fields, and obtaining a fully rigorous solution remains a major challenge. One natural question is whether one can give a probabilistic definition of these operators in the continuum using CLE or its variants and then analyze them rigorously. Note that the connectivity operator we consider in this paper is not a primary field in the minimal model CFTs of the Ising or $3$-state Potts models, as the scaling dimensions differ. Aside from Virasoro symmetry, the 3-state Potts model also enjoys $W$-symmetry~\cite{Fateev:1987vh}. Discovering this $W$-symmetry from a probabilistic perspective is also an interesting question.
    
    \item Another natural question of the $3$-state Potts model is to describe the scaling limit of all interfaces between different spin clusters. One can still use the continuum analog of the Edwards--Sokal coupling to define all these interfaces: first sample a nested CLE$_{24/5}$ (for $q = 3$, $\kappa = 10/3$), then color each cluster independently in one of three colors and look at the spin interfaces between different colors. The interfaces exhibit bifurcations and web structures. The framework of~\cite{MSW2017} is powerful for describing interfaces between a given spin cluster and the aggregate of the other spins. One natural question is whether there is a systematic way to describe all these interfaces. Answering this question would help understand the scaling limits of web models, which also exhibit bifurcations and web structures; see e.g.~\cite{Lafay-web} and the references therein.

\end{itemize}

\noindent {\bf Organization of the paper.} In Section~\ref{sec:description}, we provide background on the fuzzy Potts model and its continuum limit. In Section~\ref{sec:sle-loop}, we prove Theorem~\ref{thm: main} up to the explicit evaluation of $\mathsf{C}(\kappa)$. In Section~\ref{sec:computation}, we prove Proposition~\ref{prop:cr-moment-intro} and compute $\mathsf{C}(\kappa)$.

\medskip
\noindent {\bf Acknowledgement.} We thank Xin Sun for enlightening discussions. We thank Raoul Santachiara for his encouragement. B.W.\ is partially supported by the National Key R\&D Program of China (No.\ 2023YFA1010700). G.C. and H.L.\ are partially supported by the National Key R\&D Program of China (No.\ 2023YFA1010700) and by the Fundamental Research Funds for the Central Universities, Peking University. Z.Z.\ is partially supported by NSF grant DMS-1953848.

 Part of this project was carried out while B.W. was visiting the Hausdorff Research Institute
 for Mathematics (HIM) in 2025,  funded by the Deutsche Forschungsgemeinschaft (DFG, German Research Foundation) under Germany's Excellence Strategy – EXC-2047/1 – 390685813.

 \section{Loop ensemble description of the fuzzy Potts model}\label{sec:description}

This section reviews the fuzzy Potts model from discrete and continuous perspectives. We begin by recalling the Potts model, FK percolation, and the Edwards-Sokal coupling in Section~\ref{subsec:edwards-sokal}. Section~\ref{subsec:pre-bcle} provides the necessary background on CLE and BCLE, which is then used in Section~\ref{subsec:construct-interface} to describe the construction of continuum fuzzy Potts interfaces.

\subsection{Critical $q$-state Potts model and fuzzy Potts model}\label{subsec:edwards-sokal}

We first review the definition of the $q$-state Potts model for $q \in \mathbb{N}$. Let $(\mathbb Z^2,E(\mathbb Z^2))$ be the planar graph with vertex set $\mathbb Z^2$ and edges between nearest neighbors; we will then simply denote this graph by $\mathbb Z^2$. Given a subgraph $G=(V,E)$ of $\mathbb Z^2$, a spin configuration on $G$ is an element $\sigma \in \{1,\ldots,q\}^V$, to which we associate the Hamiltonian with free boundary conditions
\[ H_{G,q}(\sigma)=-\sum_{u,v \in V, \{u,v\} \in E} \mathbf{1}_{\sigma_u=\sigma_v}. \]
For $\beta \ge 0$, the $q$-state Potts model on $G$ with free boundary conditions at inverse temperature $\beta$ is the Gibbs measure on $\{1,\ldots,q\}^V$ given by
\begin{equation}\label{eq:def-potts}
    \mu_{G,\beta,q}(\sigma):=\frac{1}{\mathcal{Z}_{G,\beta,q}} \exp(-\beta H_{G,q}(\sigma)),
\end{equation}
where $\mathcal{Z}_{G,\beta,q}$ is the normalizing constant so that $\mu_{G,\beta,q}$ is a probability measure.

A natural way to study the Potts model is through its coupling with FK percolation. For a subgraph $G=(V,E)$ of $\mathbb Z^2$, an edge configuration on $G$ is an element $\omega \in \{0,1\}^E$, where an edge $e \in E$ is said to be open if $\omega_e=1$, and closed otherwise. A configuration $\omega$ can be viewed as a subgraph of $G$ with vertex set $V$ and edge set $\{e \in E: \omega_e=1\}$, and we denote by $o(\omega)$ and $k(\omega)$ the number of open edges and connected components of the graph, respectively. For $p \in [0,1]$ and $q>0$, the FK percolation on $G$ with cluster weight $q$, edge weight $p$ and free boundary conditions is the probability measure on $\{0,1\}^E$ given by
\[ \phi_{G,p,q}(\omega):=\frac{1}{\mathsf Z_{G,p,q}} (\tfrac{p}{1-p})^{o(\omega)} q^{k(\omega)}, \]
where $\mathsf Z_{G,p,q}$ is the normalizing constant.

The Edwards-Sokal coupling~\cite{Edwards-Sokal} unifies the Potts model and FK percolation via a joint distribution on edge-spin configurations $(\omega,\sigma)$. For $q \in \mathbb{N}$ and $p=1-e^{-\beta}$, define the probability measure on $\{1,\ldots,q\}^V \times \{0,1\}^E$ by
\[ \nu(\omega,\sigma):=\frac{1}{Z} (\tfrac{p}{1-p})^{o(\omega)} \mathbf 1_A, \]
where $Z$ is a normalizing constant and $A$ is the event that $\sigma_u=\sigma_v$ for any edge $e=\{u,v\}$ with $\omega_e=1$. The marginal of $\nu$ on spins is the Potts measure $\mu_{G,\beta,q}$, and its marginal on edges is the FK measure $\phi_{G,p,q}$. Consequently, the Potts measure $\mu_{G,\beta,q}$ can be recovered by first sampling an edge configuration $\omega$ from $\phi_{G,p,q}$ and then assigning a uniform spin from $\{1,\ldots,q\}$ to each connected component of $\omega$ independently.

Let $\partial V=\{v \in V:\deg_G(v)<4\}$ denote the vertex boundary. For a non-empty subset $\xi \subseteq \{1,\ldots,q\}$, the Potts model with boundary conditions $\xi$, denoted $\mu_{G,\beta,q}^\xi(\sigma)$, is defined by restricting the measure $\mu_{G,\beta,q}$ in~\eqref{eq:def-potts} to spin configurations satisfying $\sigma_v \in \xi$ for all $v \in \partial V$. When $\xi=\{i\}$ is a singleton, the Potts model with monochromatic boundary conditions $\{i\}$ also admits a coupling to FK percolation, now with \emph{wired} boundary conditions; see~\cite{Grimmett-rcm,BDC-notes,DC-notes} for details.

For $i \in \{1,\ldots,q\}$, as the subgraph $G$ tends to $\mathbb Z^2$, the measures $\mu_{G,\beta,q}^{\{i\}}$ converge weakly to a Gibbs measure $\mu_{\mathbb Z^2,\beta,q}^{\{i\}}$, called the \emph{infinite-volume} Potts measure with monochromatic boundary conditions $\{i\}$. For $q \ge 1$, the model undergoes a phase transition at the critical inverse temperature $\beta_c(q)=\log(1+\sqrt{q})$~\cite{BDC12}: below $\beta_c$, we have $\mu_{\mathbb Z^2,\beta,q}^{{i}}[\sigma_0=i]=1/q$, while above it, $\mu_{\mathbb Z^2,\beta,q}^{{i}}[\sigma_0=i]>1/q$. Moreover, this phase transition is continuous for $q \in [1,4]$~\cite{DCST17}. The critical $q$-state Potts model refers to the case $\beta=\beta_c(q)$, and the corresponding FK$_q$ percolation has its critical point at $p_c(q)=1-e^{-\beta_c(q)}=\sqrt{q}/(1+\sqrt{q})$.

The Edwards-Sokal coupling motivates the definition of the fuzzy Potts model, a generalization of the Potts model where vertices are colored red or blue instead of assigned one of $q$ spins. In this paper, we focus on the case where we start with a \emph{critical} FK percolation. For a subgraph $G=(V,E)$ of $\mathbb Z^2$, a vertex configuration on $G$ is an element $\tau \in \{R,B\}^V$, where $\tau_v=R$ (resp.\ $\tau_v=B$) indicates that the vertex $v$ is red (resp.\ blue). For $q>0$ (not necessarily an integer) and $r \in (0,1)$, the fuzzy Potts model on $G$ with parameters $(q,r)$ and free boundary conditions is defined by the following two-step procedure:
\begin{enumerate}[(1)]
    \item Sample an edge configuration $\omega$ from the critical FK$_q$ measure $\phi_{G,p_c,q}$ with free boundary conditions.
    \item Independently color each connected component of $\omega$ in red with probability $r$ and in blue with probability $1-r$; the model is the marginal distribution on the vertex colors.
\end{enumerate}

The fuzzy Potts measure with red (resp.\ blue) boundary conditions is defined analogously using the critical FK$_q$ measure with wired boundary conditions, with the additional constraint that any boundary-touching cluster is forced to be red (resp.\ blue); see~\cite{KSL22,LSYZ24} for details.

When $q \in \mathbb{N}$ and $r=k/q$ for some $k \in \{1,\ldots,q-1\}$, the fuzzy Potts measure coincides with the law of the two-color projection of a critical $q$-state Potts configuration $\sigma$. Specifically, if we set $\tau_v=R$ for $\sigma_v \in \{1,\ldots,k\}$ and $\tau_v=B$ otherwise, the resulting distribution on $\tau$ is exactly the fuzzy Potts model with parameters $(q, r)$.

\subsection{Preliminaries of CLE, BCLE, and CLE percolations}\label{subsec:pre-bcle}

In this section, we recall the key definitions and properties of the \emph{conformal loop ensemble} (CLE) and its variant, the \emph{boundary conformal loop ensemble} (BCLE).

For $\kappa \in (8/3,8)$, the non-nested CLE$_\kappa$ is a conformally invariant random collection of non-crossing loops in which no loop surrounds another~\cite{shef-cle,shef-werner-cle}. Each loop is a Schramm's SLE$_\kappa$-type curve. When $\kappa \in (8/3,4]$, the loops are almost surely simple and does not intersect the boundary of the domain or each other. For $\kappa \in (4,8)$, the loops are nonsimple and may touch (but not cross) themselves and each other. CLE$_\kappa$ can be constructed via the SLE$_\kappa(\kappa-6)$ branching tree~\cite{SW05,shef-cle}; additionally, for $\kappa \in (8/3,4]$, it admits a Brownian loop soup construction~\cite{shef-werner-cle}. A nested version of CLE$_\kappa$ is constructed by iterating the non-nested CLE$_\kappa$ process within each simply connected component of the complement of the loop ensemble.

It is widely believed that CLE describes the scaling limit of interfaces in various critical statistical mechanics models. Specifically, for $q \in \{2,3\}$, the collection of outermost interfaces in a critical $q$-state Potts model with monochromatic boundary conditions is conjectured to converge to a non-nested simple CLE$_\kappa$. Furthermore, for $q \in (0,4]$, the interfaces of critical FK$_q$ percolation are expected to converge to a nested CLE$_{\kappa'}$ with $\kappa' \in [4,8)$.  
So far, these conjectures have only been verified for the Ising and FK-Ising models ($q=2$, $\kappa=3$ and $\kappa'=16/3$)~\cite{BH19,Smi10,KS16,KS19}.

For $\kappa \in (2,4]$ and $\rho \in (-2,\kappa-4)$ or $\kappa \in (4,8)$ and $\rho \in (\kappa/2-4,\kappa/2-2)$, the BCLE$_\kappa(\rho)$ is a conformally invariant collection of boundary-touching loops in simply connected domains~\cite{MSW2017}. Its construction is similar to that of CLE$_\kappa$ but uses an SLE$_\kappa(\rho;\kappa-6-\rho)$ branching tree $\mathcal T$ instead of SLE$_\kappa(\kappa-6)$. The branches of $\mathcal T$ are naturally oriented from the root towards all other boundary points, inducing an orientation on the boundaries of the complementary components that form either clockwise or counterclockwise loops. The collection of clockwise loops defines $\BCLE^\clockwise_\kappa(\rho)$; these are referred to as the true loops. The remaining components, not surrounded by a true loop, have boundaries that form counterclockwise loops, which are called the false loops of $\BCLE^\clockwise_\kappa(\rho)$. By reversing the orientation of every loop in $\BCLE^\clockwise_\kappa(\rho)$, we obtain the collection $\BCLE^\counterclockwise_\kappa(\rho)$ of counterclockwise loops, whose false loops are clockwise. Equivalently, $\BCLE^\counterclockwise_\kappa(\rho)$ has the same law as the false loops of $\BCLE^\clockwise_\kappa(\kappa-6-\rho)$.

BCLEs are expected to describe the scaling limit of critical models with special boundary conditions. For a fixed $\kappa$, each parameter $\rho$ corresponds to a distinct boundary condition, which can be interpreted as weighting boundary-touching loops relative to interior loops.

For $\beta \in (-1,1)$, the labeled CLE$_\kappa^\beta$ is an oriented version of the non-nested CLE$_\kappa$ in which each loop is independently oriented counterclockwise with probability $(1+\beta)/2$ and clockwise with probability $(1-\beta)/2$. The following CLE percolation result, which provides an iterative construction of labeled CLE using BCLEs, is a combination of Theorems 7.2 and 7.7 in~\cite{MSW2017}.

\begin{theorem}\label{thm:CLE-percolation}
     For each $\kappa \in (2,4)$ and $\beta \in (-1,1)$, there exists $\rho=\rho(\beta,\kappa) \in (-2,\kappa-4)$ such that the following holds. Let $\kappa'=16/\kappa \in (4,8)$ and define
\begin{equation}\label{eq:parameter-nonsimple}
    \rho_R'=-\frac{\kappa'}{2}-\frac{\kappa'}{4} \rho, \quad \rho_B'=\kappa'-4+\frac{\kappa'}{4} \rho.
\end{equation}
The labeled $\CLE_{\kappa'}^\beta$ in a simply connected domain $D \subsetneq \mathbb C$ can then be constructed by iterating $\BCLE^\clockwise_\kappa(\rho)$, $\BCLE^\counterclockwise_{\kappa'}(\rho_R')$, and $\BCLE^\clockwise_{\kappa'}(\rho_B')$. Specifically, starting with $\Gamma^\clockwise = \Gamma^\counterclockwise = \emptyset$, the iteration proceeds as follows:
\begin{enumerate}[label=(\arabic*)]
    \item Sample $\Lambda \sim \BCLE^\clockwise_\kappa(\rho)$ in $D$.
    \item In the domains enclosed by clockwise true loops (resp.\ counterclockwise false loops) of $\Lambda$, independently sample $\BCLE^\counterclockwise_{\kappa'}(\rho_R')$ (resp.\ $\BCLE^\clockwise_{\kappa'}(\rho_B')$). Add the counterclockwise true loops of $\BCLE^\counterclockwise_{\kappa'}(\rho_R')$ to $\Gamma^\counterclockwise$, and the clockwise true loops of $\BCLE^\clockwise_{\kappa'}(\rho_B')$ to $\Gamma^\clockwise$.
    \item Iterate the previous two steps independently in every simply connected domain not enclosed by any loop in $\Gamma^\clockwise \cup \Gamma^\counterclockwise$. (These domains correspond to the interiors of false loops of $\BCLE_{\kappa'}^{\counterclockwise}(\rho_R')$ or $\BCLE_{\kappa'}^{\clockwise}(\rho_B')$ in the previous step.)
\end{enumerate}
The resulting collection $\Gamma = \Gamma^\clockwise \cup \Gamma^\counterclockwise$ is equal in law to a labeled $\CLE_{\kappa'}^\beta$ on $D$. Moreover, the parameter $\rho$ is determined by $\beta$ and $\kappa$ through the relation
\begin{equation}\label{eq:relation-nonsimple}
    \frac{1-\beta}{2}=\frac{\sin(\pi \rho/2)}{\sin(\pi \rho/2)-\sin(\pi (\kappa-\rho)/2)}.
\end{equation}
\end{theorem}

The relation~\eqref{eq:relation-nonsimple} was established in~\cite{MSW21-nonsimple}; an alternative proof based on the winding probability of BCLEs is given in Section 2.2 of~\cite{LSYZ24}.

Theorem~\ref{thm:CLE-percolation} can be viewed as the continuous analog of the Edwards-Sokal coupling between the $q$-state Potts model and the FK$_q$ percolation for $q \in (0,4)$. For $q \in [1,4)$, under the assumption that FK$_q$ percolation interfaces converge to nested CLE$_{\kappa'}$ (where $\kappa'=4\pi/\arccos(-\sqrt{q}/2) \in (4,6]$), it was shown in~\cite[Theorem 4.2]{KSL22} (see also Section 2.6 therein) that the scaling limit of the fuzzy Potts (with red probability $r=\frac{1+\beta}{2}$) interfaces coincides with the collection of $\BCLE_\kappa$-type loops constructed in Theorem~\ref{thm:CLE-percolation}. This limiting object---a conformally invariant collection of loops---is referred to as the continuum fuzzy Potts model on the simply connected domain $D$. Its whole-plane variant is defined as the weak limit of this model as $D$ tends to $\mathbb{C}$, analogous to the definition of whole-plane CLE in~\cite{KW16}.

\subsection{Description of spin interfaces via BCLE}\label{subsec:construct-interface}

We now describe the construction of continuum fuzzy Potts interfaces for the special case $r=1/q$ (which corresponds to ordinary Potts model) using Theorem~\ref{thm:CLE-percolation}. For the purpose of this paper, we focus on the countable collection of interfaces that separate a fixed point (say, the origin) from $\infty$. Label these loops as $(\mathcal L^k)_{k \in \mathbb N}$ (or $(\mathcal L^k)_{k \in \mathbb Z}$ in the whole-plane case) such that $\mathcal L^{k+1}$ is surrounded by $\mathcal L^k$. Each $\mathcal{L}^k$ is either a red-blue interface (red on the outer side and blue on the inner) or a blue-red interface (blue on the outer side and red on the inner), and these types alternate with $k$.

In the following, we will describe the conditional law of $\mathcal L^{k+1}$ given $\mathcal L^k$. Observe that $\mathcal L^{k+1}$ can be interpreted as an outermost interface in a continuum fuzzy Potts model on the domain enclosed by $\mathcal L^k$ under red or blue boundary conditions (as explained in the proof of Theorem 7.10 of~\cite{MSW2017}). By conformal invariance, it suffices to characterize the law of the following two loops:
\begin{enumerate}[(a)]
    \item The outermost red-blue interface surrounding the origin in a continuum fuzzy Potts model on $\mathbb{D}$ with red boundary conditions, denoted by $\mathcal L_{R \to B}$.
    \item The outermost blue-red interface surrounding the origin in a continuum fuzzy Potts model on $\mathbb{D}$ with blue boundary conditions, denoted by $\mathcal L_{B \to R}$.
\end{enumerate}

Let $q \in (1,4)$, $\kappa=4\arccos(-\sqrt{q}/2)/\pi \in (8/3,4)$, and $\kappa'=16/\kappa$. Define $\rho=\frac{3\kappa}{2}-6$ which solves~\eqref{eq:relation-nonsimple} with $r=\frac{1+\beta}{2}=\frac{1}{q}$, and let $\rho_B'=\kappa'-4+\frac{\kappa'}{4} \rho=2-\frac{\kappa'}{2}$ as in~\eqref{eq:parameter-nonsimple}. By Theorem~\ref{thm:CLE-percolation}, the outermost interface $\mathcal L_{B \to R}$ can be explored as follows:\footnote{We remind the readers that this exploration is different from that described in~\cite{LSYZ24}, since we are working on monochromatic boundary conditions rather than free boundary conditions. In this case, we first explore the boundary-touching FK clusters, whose boundaries correspond to $\BCLE_{\kappa'}(\rho_B')$ in the continuum.}

\noindent \textit{Step 1.} Sample $\Xi_0' \sim \BCLE_{\kappa'}^\clockwise(\rho_B')$ in $D_0:=\mathbb{D}$, and define the domain $D_1'$ according to two cases:
\begin{itemize}
    \item If the origin is enclosed by a true loop $\eta_0'$ of $\Xi_0'$, sample a non-nested CLE$_{\kappa'}$ process $\Gamma_0'$ in the connected component of $D_0 \setminus \eta_0'$ containing the origin. The origin is a.s. enclosed by a unique loop $\widetilde \eta_0$ of $\Gamma_0'$. Define $D_1'$ to be the connected component of $D_0 \setminus \widetilde\eta_0$ containing the origin.
    \item If the origin is enclosed by a false loop $\eta_0'^*$, define $D_1'$ to be the connected component of $D_0 \setminus \eta_0'^*$ containing the origin.
\end{itemize}
\noindent \textit{Step 2.} Sample $\Xi_1 \sim \BCLE^\clockwise_\kappa(\rho)$ in $D_1'$, and proceed as follows:
\begin{itemize}
    \item If the origin is enclosed by a true loop $\eta_1$ of $\Xi_1$, set $\mathcal L_{B \to R}=\eta_1$ and \textit{stop}.
    \item If the origin is enclosed by a false loop $\eta_1^*$, define $D_1$ to be the simply connected component of $D_1' \setminus \eta_1^*$ containing the origin. Return to \textit{Step 1} and continue exploring within $D_1$ instead of $D_0$, increasing all indices by 1.
\end{itemize}

Since the termination probability $\mathbb{P}[0 \in \BCLE^\clockwise_\kappa(\rho)]$ is strictly positive, the exploration almost surely terminates after finitely many iterations, and the interface $\mathcal{L}_{B \to R}$ is almost surely well-defined.

The outermost interface $\mathcal L_{R \to B}$ can be constructed analogously, but using instead $\rho=-\frac{\kappa}{2}$ which satisfies~\eqref{eq:relation-nonsimple} for $r=\frac{1+\beta}{2}=1-\frac{1}{q}$ (due to color switching). Alternatively, by combining~\cite[Theorem 7.10]{MSW2017} and~\cite[Theorem 1.2]{MSW21-nonsimple}, under red boundary conditions, the red clusters touching the boundary form a $\CLE_\kappa$ carpet, and thus $\mathcal L_{R \to B}$ is exactly the unique loop in the non-nested $\CLE_\kappa$ that surrounds the origin.

\section{Proof of Theorem~\ref{thm: main} via loop equivalence}\label{sec:sle-loop}

Suppose $\omega$ is sampled from a whole-plane continuum fuzzy Potts model with parameters $q \in (1,4)$ and $r=1/q$. Let $\widetilde{\SLE}_\kappa^\lp$ be the law of the loop sampled from the counting measure on the collection of blue-red interfaces. Let $\SLE_\kappa^\lp$ be the law of the loop sampled from the counting measure on the full-plane $\CLE_\kappa$. Then define $\widetilde{\SLE}_\kappa^\sep$, $\SLE_\kappa^\sep$ to be the restrictions of $\widetilde{\SLE}_\kappa^\lp$ and $\SLE_\kappa^\lp$ to the loops separating $0$ and $\infty$, respectively. According to~\cite{KW16}, the measure $\SLE_\kappa^\sep$ is the same as the $\SLE_\kappa$ loop measure defined in~\cite{zhan-loop-measures}, restricted to loops separating $0$ and $\infty$ (see also~\cite{ACSW-loop} for another proof using couplings of SLE and Liouville quantum gravity).

The main result of this section is the following up-to-constant equivalence of $\widetilde{\SLE}_\kappa^\sep$ and $\SLE_\kappa^\sep$.
\begin{theorem}\label{thm:ratio}
We have $\wt\SLE_\kappa^\sep=\mathsf{C}(\kappa)^{-2} \,\SLE_\kappa^\sep$, with the positive constant $\mathsf{C}(\kappa)$ given by
\begin{equation}\label{eq:c-kappa}
    \mathsf{C}(\kappa):=\sqrt{\frac{\mathbb{E}[\log \mathsf{R}_{R\to B}]+\mathbb{E}[\log \mathsf{R}_{B\to R}]}{\mathbb{E}[\log \mathsf{R}_{R\to B}]}}.
\end{equation}
\end{theorem}

We will further derive the explicit value of $\mathsf{C}(\kappa)$ in Section~\ref{sec:computation}.
\begin{theorem}\label{thm:value}
    We have $\mathsf{C}(\kappa)=\sqrt{\kappa/2}\cdot\frac{\sin(\kappa\pi/2)}{\sin(4\pi/\kappa)}$ for $\kappa \in (8/3,4)$.
\end{theorem}

We will first explain in Section~\ref{subsec:proof-1.2} that Theorems~\ref{thm:ratio} and~\ref{thm:value} readily imply Theorem~\ref{thm: main}. Then we prove Theorem~\ref{thm:ratio} in the remainder of this section, following the roadmap in~\cite{ACSW-loop}. The proof of Theorem~\ref{thm:value} will be finished in Section~\ref{sec:computation} using the results from~\cite{LSYZ24}.

\subsection{Proof of Theorem~\ref{thm: main} given Theorems~\ref{thm:ratio} and~\ref{thm:value}}\label{subsec:proof-1.2}

\begin{proof}[Proof of Theorem~\ref{thm: main}]
We aim to relate~\eqref{eq:def-gf} to the CLE cluster Green's function $G_n^{\CLE}(z_1,\ldots,z_n)$ introduced in~\cite{ACSW24}. Let $\eta \subset \mathbb{C}$ be a simple loop. Define $D(\eta)$ to be the finite connected component of $\mathbb{C} \setminus \eta$, and $\CLE_\kappa^{D(\eta)}$ to be the law of $\CLE_\kappa$ in $D(\eta)$. For a $\CLE_\kappa$ process $\Gamma$, let $\mu_\Gamma(dz)$ be its Miller-Schoug measure on its carpet.
Then, according to~\cite[Equation (1.6)]{ACSW24}, $G_n^{\CLE}(z_1,\ldots,z_n)$ satisfies
\begin{equation}\label{eq:n-point-2}
    G_n^{\CLE}(z_1,\ldots,z_n)dz_1 \cdots dz_n=\int \prod_{i=1}^n\mu_\Gamma(dz_i)\CLE_\kappa^{D(\eta)}(d\Gamma) \SLE_\kappa^\lp(d\eta).
\end{equation}

We claim that $G_n^\potts(z_1,\ldots,z_n)=\mathsf{C}(\kappa)^{-2}G_n^{\rm CLE}(z_1,\ldots,z_n)$ for any positive integer $n$, where $\mathsf{C}(\kappa)$ is defined as in Theorem~\ref{thm:ratio}.
Indeed, since $\wt\SLE_\kappa^\lp$ is the counting measure of blue-red interfaces, by~\eqref{eq:def-gf} we have
    \begin{equation}\label{eq:n-point-1}
        G_n^\potts(z_1,\ldots,z_n)dz_1 \cdots dz_n=\int \prod_{i=1}^n\mu_\Gamma(dz_i)\CLE_\kappa^{D(\eta)}(d\Gamma) \wt\SLE_\kappa^\lp(d\eta).
    \end{equation}
For $z \in \mathbb{C}$, define $\mathcal{J}_z$ to be the collection of simple loops on $\mathbb{C}$ that separate $0$ and $\infty$ (hence $\widetilde{\SLE}_\kappa^\sep$ and $\SLE_\kappa^\sep$ are the restrictions of $\widetilde{\SLE}_\kappa^\lp$ and $\SLE_\kappa^\lp$ to $\mathcal{J}_0$, respectively). By Theorem~\ref{thm:ratio} and conformal invariance, we have for any $z \in \mathbb{C}$, $\widetilde{\SLE}_\kappa^\lp\big|_{\mathcal{J}_z}=\mathsf{C}(\kappa)^{-2}\SLE_\kappa^\lp\big|_{\mathcal{J}_z}$. By varying $z$, we find $\wt\SLE_\kappa^\lp=\mathsf{C}(\kappa)^{-2}\SLE_\kappa^\lp$. Therefore, the claim follows from comparing~\eqref{eq:n-point-1} and~\eqref{eq:n-point-2}.
    
    Finally, by~\cite[Theorem 1.4]{ACSW24}, we have
    $$\frac{G_3^{\rm CLE}(z_1,z_2,z_3)}{\sqrt{G_2^{\rm CLE}(z_1,z_2)G_2^{\rm CLE}(z_2,z_3)G_2^{\rm CLE}(z_1,z_3)}}=C_{\beta=\frac{2}{\sqrt{\kappa}}}^{\rm ImDOZZ}\left(\frac{1}{4\beta} - \frac{\beta}{2},\frac{1}{4\beta} - \frac{\beta}{2},\frac{1}{4\beta} - \frac{\beta}{2}\right).$$
    Combined with the above claim as well as the explicit value of $\mathsf{C}(\kappa)$ in Theorem~\ref{thm:value}, we obtain that
    $$\frac{G_3^{\rm Potts}(z_1,z_2,z_3)}{\sqrt{G_2^{\rm Potts}(z_1,z_2)G_2^{\rm Potts}(z_2,z_3)G_2^{\rm Potts}(z_1,z_3)}}=\mathsf{C}(\kappa) \cdot C_{\beta=\frac{2}{\sqrt{\kappa}}}^{\rm ImDOZZ}\left(\frac{1}{4\beta} - \frac{\beta}{2},\frac{1}{4\beta} - \frac{\beta}{2},\frac{1}{4\beta} - \frac{\beta}{2}\right),$$
    as desired.
\end{proof}

\subsection{Markov chain of interfaces}\label{subsec:mc}

In the remainder of this section, we will prove Theorem~\ref{thm:ratio} following the framework of~\cite[Section 6]{ACSW-loop}. For convenience, we instead consider the horizontal cylinder $\cC$ obtained from $\mathbb{R}\times[0,2\pi]$ by identifying $(x,0)$ and $(x,2\pi)$ for $x \in \mathbb R$.
Let $\Loop(\mathcal{C})$ denote the set of simple loops on $\mathcal{C}$ that separates $\pm\infty$. For $\eta \in \Loop(\mathcal{C})$, define $\mathcal{C}_\eta^+$ as the connected component of $\mathcal{C} \setminus \eta$ containing $+\infty$. Furthermore, let $\Loop_0(\mathcal{C})$ be the subset of $\Loop(\mathcal{C})$ consisting of loops $\eta$ with $\max\{\mathrm{Re}(z): z \in \eta\}=0$.

We now define a Markov chain of (shifted) blue-red interfaces on $\Loop_0(\mathcal{C})$ as follows. Given $\eta^0 \in \Loop_0(\mathcal{C})$, sample a CLE$_\kappa$ on $\mathcal{C}_{\eta^0}^+$, and let $\eta^*$ be its outermost loop surrounding $+\infty$\footnote{As noted in Section~\ref{subsec:construct-interface}, $\eta^*$ can be interpreted as an outermost red-blue interface of a continuum fuzzy Potts model with red boundary conditions.}. Then, independently sample a continuum fuzzy Potts model with parameters $(q,r)$ and blue boundary conditions on $\mathcal{C}_{\eta^*}^+$, and denote its outermost blue-red interface surrounding $+\infty$ by $\widetilde{\eta}^1$. Finally, translate $\widetilde{\eta}^1$ to obtain an element $\eta^1\in\Loop_0(\mathcal{C})$. Iterating this procedure defines a Markov chain $(\eta^i)_{i \ge 0}$ on $\Loop_0(\mathcal{C})$.

We first show that $\wt\SLE_\kappa^\sep$ (resp.\ $\SLE_\kappa^\sep$) defined at the beginning of Section~\ref{sec:sle-loop} naturally induces a probability measure on $\Loop_0(\mathcal{C})$. Let $\wt\SLE_\kappa^\sep(\mathcal{C})$ (resp.\ $\SLE_\kappa^\sep(\cC)$) be the pushforward measure of $\wt\SLE_\kappa^\sep$ (resp.\ $\SLE_\kappa^\sep$) under the conformal map $z \mapsto -\log z$ that sends $\hat{\mathbb{C}}$ to $\mathcal{C}$, 
which is an infinite measure on simple loops on $\mathcal{C}$ that separates $\pm\infty$.
\begin{lemma}\label{lem:shape}
    For a loop $\eta$ sampled from $\wt\SLE_\kappa^\sep(\mathcal{C})$ (resp.\ $\SLE_\kappa^\sep(\cC)$), there is a unique decomposition $\eta = \eta^0 + \mathbf{t}$ with $\eta^0 \in \Loop_0(\mathcal{C})$ and $\mathbf{t} \in \mathbb{R}$, and there exists a probability measure $\wt{\mathcal{L}}_\kappa(\cC)$ (resp.\ $\mathcal{L}_\kappa(\cC)$) on $\Loop_0(\mathcal{C})$ and constants $\wt C,C>0$ such that the joint law of $(\eta^0,\bf t)$ equals $\wt C\wt{\mathcal{L}}_\kappa(\cC)\times dt$ (resp.\ $C\mathcal{L}_\kappa(\cC)\times dt$). Moreover, $\wt C=(\mathbb{E}[\log \mathsf{R}_{R\to B}]+\mathbb{E}[\log \mathsf{R}_{B\to R}])^{-1}$ and $C=(\mathbb{E}[\log \mathsf{R}_{R\to B}])^{-1}$.
\end{lemma}

\begin{proof}
    The first claim follows directly from the translation invariance of $\wt\SLE_\kappa^\sep(\mathcal{C})$ (resp.\ $\SLE_\kappa^\sep(\cC)$). The constant $C=(\mathbb{E}[\log \mathsf{R}_{R\to B}])^{-1}$ is derived in~\cite[Proposition 9.1]{ACSW-loop}, and it is straightforward to see that the same argument applies to $\wt\SLE_\kappa^\sep(\mathcal{C})$, yielding the expression for $\widetilde C$.
\end{proof}

The following proposition shows that $\wt{\mathcal{L}}_\kappa(\cC)$ is the unique stationary measure of the Markov chain defined above.

\begin{proposition}\label{prop:markov}
    Let $(\eta^i)_{i \ge 0}$ be the Markov chain starting from $\eta^0$. Then $\eta^n$ converges in the total variation distance to $\widetilde{\mathcal{L}}_\kappa(\mathcal{C})$. Moreover, $\widetilde{\mathcal{L}}_\kappa(\mathcal{C})$ is the unique stationary measure of the Markov chain.
\end{proposition}

\begin{proof}
Suppose that $D \subset \mathbb{C}$ is a bounded simply connected domain containing the origin, and $\omega_D$ is sampled from a continuum fuzzy Potts model on $D$ with parameters $(q,r)$ and red boundary conditions. Let $\{\eta_n^D\}_{n \ge 0}$ be the collection of blue-red interfaces in $\omega_D$ surrounding the origin, ordered so that $\eta_{n+1}^D$ is surrounded by $\eta_n^D$, with the convention that $\eta_0^D:=\partial D$. Recall that $\omega$ is sampled from the whole-plane continuum fuzzy Potts model.
In what follows, we construct a coupling between $\omega_D$ and $\omega$ such that, with positive probability, the sequence $\{\eta_n^D\}$ coincides with the blue-red interfaces in $\omega$ after finitely many steps. The desired result is a direct consequence of this coupling.

We achieve this by the continuum Edwards-Sokal coupling introduced in Theorem~\ref{thm:CLE-percolation}.
Suppose $\Gamma_D$ (resp.\ $\Gamma$) is a $\CLE_{\kappa'}$ process on $D$ (resp.\ on the whole-plane). Let $\eta_0'^D$ be the outermost loop of $\Gamma_D$ surrounding the origin, and let $\eta_0'$ be the outermost loop of $\Gamma$ contained in $D$ and surrounds the origin. According to~\cite[Lemma 6.14]{ACSW-loop}, there is a coupling of $\Gamma_D$ and $\Gamma$ such that $\eta_0'^D$ and $\eta_0'$ coincide and do not touch $\partial D$ with a positive probability. For a non-crossing loop $\eta \subset \mathbb{C}$ surrounding the origin, denote by $D(\eta)$ the connected component of $\mathbb{C} \setminus \eta$ containing the origin. Then, by the Markov property of $\CLE_{\kappa'}$, this gives a coupling of $(\Gamma_D,\Gamma)$ such that with a positive probability, $(\eta_0'^D,D(\eta_0'^D),\Gamma_D|_{D(\eta_0'^D)})$ and $(\eta_0',D(\eta_0'),\Gamma|_{D(\eta_0')})$ coincide.

Suppose we are now under the event that the above coupling of $(\Gamma_D,\Gamma$) succeeds, and we color $(\Gamma_D,\Gamma)$ in order to obtain a coupling of $(\omega_D,\omega)$. Namely, we color each loop in $\Gamma_D$ that touches $\partial D$ red. Each remaining loop of $\Gamma_D$ is independently colored red (resp.\ blue) with probability $r$ (resp.\ $1-r$). Next, we assign to each loop in $\{\eta_0'\}\cup\Gamma|_{ D(\eta_0'^D)}$ the same color as its counterpart in $\{\eta_0'^D\}\cup\Gamma_D|_{D(\eta_0'^D)}$ (recall that they are coupled to be the same). The remaining loops in $\Gamma$ are independently colored red or blue with probability $r$ or $1-r$. Note that the marginal law of colors of loops in $\Gamma$ is still independently red or blue with probability $r$ or $1-r$, respectively, and hence this gives a coupling between $\omega_D$ and $\omega$. Now, on the event that the coupling of $(\Gamma_D,\Gamma)$ succeeds, consider exploring all blue-red interfaces of $\omega_D$ from $\partial D$ to the origin. These interfaces almost surely enter the domain $D(\eta_0'^D)$ after a finite number of steps, and they will coincide with the corresponding blue-red interfaces of $\omega$ thereafter, thus providing the desired coupling.
\end{proof}

According to Proposition~\ref{prop:markov}, to show Theorem~\ref{thm:ratio}, it remains to show that $\mathcal{L}_\kappa(\mathcal{C})$ is also a stationary measure of the Markov chain. This can be done by coupling the SLE loop measure with the Liouville quantum gravity (LQG) surfaces. In the following Section~\ref{sec:LQG}, we will first recall some LQG backgrounds, and then finish the proof of Theorem~\ref{thm:ratio} in Section~\ref{subsec:shape}.

\subsection{Liouville quantum gravity surfaces}\label{sec:LQG}

This subsection reviews some basic geometric concepts in LQG and conformal welding. Seasoned readers may skim or proceed directly to Proposition~\ref{prop:weld-sle-loop}.

We begin by reviewing the Gaussian free field (GFF) on the horizontal strip $\mathcal{S}=\mathbb{R} \times (0,\pi)$. Let $m$ be the uniform measure on $\{0\} \times (0,\pi)$. Define the Dirichlet inner product $\langle f,g\rangle_\nabla = (2\pi)^{-1} \int_\mathcal{S} \nabla f\cdot\nabla g $ on the space $\{f\in C^\infty(\mathcal{S}):\int_\mathcal{S}|\nabla f|^2<\infty \mbox{ and } \int_\mathcal{S} f(z)m(\dd z)=0\}$, and let $H(\mathcal{S})$ be its Hilbert space closure under $\langle \cdot,\cdot\rangle_\nabla$.
Let $(f_n)_{n=1}^\infty$ be an orthonormal basis of $H(\mathcal{S})$, and let $(\alpha_n)_{n=1}^\infty$ be a sequence of i.i.d.\ standard Gaussian random variables. Then the summation $h_\mathcal{S}=\sum_{n=1}^\infty \alpha_n f_n$ converges almost surely in the space of distributions. We call $h_\mathcal{S}$ a \emph{Gaussian free field} on $\mathcal{S}$ with normalization $\int_{\mathcal{S}} h_\mathcal{S}(z) m(\dd z)=0$; see~\cite[Section 4.1.4]{DMS-mating-of-tree} for details.

Fix an LQG parameter $\gamma \in (0,2)$. A $\gamma$-Liouville quantum gravity surface is defined as follows. Consider pairs $(D,h)$ where $D \subseteq \mathbb{C}$ is a domain and $h$ is a distribution on $D$. Define an equivalence relation $\sim_\gamma$ by $(D, h) \sim_\gamma (\widetilde{D}, \widetilde{h})$ if and only if there exists a conformal map $g: D \to \widetilde{D}$ such that
\begin{equation}\label{eq:equiv-rel}
    \wt h=g \bullet_\gamma h := h \circ g^{-1}+Q \log |(g^{-1})'|, \quad \mbox{where} \quad Q=\frac{\gamma}{2}+\frac{2}{\gamma}.
\end{equation}
A quantum surface is an equivalence class of pairs $(D,h)$ under $\sim_\gamma$, and an embedding is a choice of $(D,h)$ from the equivalence class. More generally, for $m,n \in \mathbb{N}$, a curve-decorated quantum surface with marked points is an equivalence class of tuples $(D,h,x_1,\ldots,x_m,\eta_1,\ldots,\eta_n)$, where $x_i \in \overline{D}$ and $\eta_j$ are curves in $\overline{D}$. We say $(D,h,x_1,\ldots,x_m,\eta_1,\ldots,\eta_n) \sim_\gamma (\widetilde D,\widetilde h,\widetilde x_1,\ldots,\widetilde x_m,\widetilde \eta_1',\ldots,\widetilde \eta_n)$ if there exists a conformal map $g:D \to \widetilde{D}$ satisfying~\eqref{eq:equiv-rel} such that $g(x_i) = \widetilde{x}_i$ and $g(\eta_j) = \widetilde{\eta}_j$ for all $i, j$.

For a $\gamma$-quantum surface $(D,h)/{\sim_{\gamma}}$ embedded as $(\mathcal{S},\phi)$, where $\phi$ is the sum of $h_{\mathcal{S}}$ and a (possibly random) function on $\overline{\mathcal{S}}$ continuous except at finitely many points, the \emph{quantum area measure} $\mu_\phi$ is defined as the weak limit of $\mu_\phi^\epsilon:=\epsilon^{\gamma^2/2} e^{\gamma \phi_\epsilon(z)} \dd^2 z$ as $\epsilon \to 0$, where $\dd^2 z$ is the Lebesgue measure on $\mathcal{S}$ and $\phi_\epsilon(z)$ is the average of $\phi$ over the circle $\partial \mathcal{B}(x,\epsilon) \cap \mathcal{S}$~\cite{DS11,SW16}. Similarly, the \emph{quantum boundary length measure} $\nu_\phi$ is given by the weak limit of $\nu_\phi^\epsilon:=\epsilon^{\gamma^2/4} e^{\frac{\gamma}{2} \phi_\epsilon(x)} \dd x$ as $\epsilon \to 0$, where for $x \in \partial \mathcal{S}$, $\phi_\epsilon(x)$ is the average of $\phi$ over the semi-circle $\partial \mathcal{B}(x,\epsilon) \cap \mathcal{S}$.

We now recall the radial-lateral decomposition of $h_{\mathcal{S}}$. The space $H(\mathcal{S})$ admits an orthogonal decomposition $H(\mathcal{S})=H_1(\mathcal{S}) \oplus H_2(\mathcal{S})$, where $H_1(\mathcal{S})$ (resp.\ $H_2(\mathcal{S})$) consists of functions in $H(\mathcal{S})$ which are constant (resp.\ have mean zero) on each vertical line $\{t\} \times (0,\pi)$ for each $t \in \mathbb{R}$. This yields a decomposition $h_{\mathcal{S}}=h_{\mathcal{S}}^1+h_{\mathcal{S}}^2$, where $h_{\mathcal{S}}^1$ and $h_{\mathcal{S}}^2$ are the projections of $h_\mathcal{S}$ onto $H_1(\mathcal{S})$ and $H_2(\mathcal{S})$, respectively, and are independent.
Moreover, the process $\{ h_{\mathcal{S}}^1(t) \}_{t \in \mathbb{R}}$ is distributed as $\{B_{2t}\}_{t \in \mathbb{R}}$, where $(B_t)_{t \in \mathbb{R}}$ is a standard two-sided Brownian motion with $B_0=0$~\cite[Section 4.1.6]{DMS-mating-of-tree}.

We now review the quantum disks and quantum spheres as defined in~\cite[Section 4.5]{DMS-mating-of-tree}, following the presentation in~\cite{AHS23,AHS21}.

\begin{definition}[Thick quantum disk]
    For $W \ge \frac{\gamma^2}{2}$, we define the measure $\mathcal{M}_2^{\rm disk}(W)$ as follows. Write $\beta=\gamma+\frac{2-W}{\gamma}$ and let $(B_t)_{t \ge 0}$ be a standard Brownian motion conditioned on $B_{2t}-(Q-\beta)t<0$ for all $t>0$, and $(\widetilde B_t)_{t \ge 0}$ be its independent copy. Let
    \begin{equation*}
        Y_t=
        \begin{cases}
            B_{2t}-(Q-\beta) t, & \mbox{ for }\,t \ge 0, \\
            \widetilde B_{-2t}+(Q-\beta)t, & \mbox{ for }\,t<0,
        \end{cases}
    \end{equation*}
    and set $h_1(z)=Y_{\mathrm{Re} z}$ for each $z \in \mathcal{S}$. Let $h_2$ be a random generalized function with the same law as $h_\mathcal{S}^2$ defined above and independent from $h_1$.
    Sample $\mathbf{c} \in \mathbb{R}$ independently from the measure $\frac{\gamma}{2} e^{(\beta-Q)c} \dd c$, and set $\phi(z)=h_1(z)+h_2(z)+\mathbf{c}$. The infinite measure $\mathcal{M}_2^{\rm disk}(W)$ is defined as the law of $(\mathcal{S},\phi,-\infty,+\infty)/{\sim_\gamma}$.
\end{definition}

For $W=2$, the two marked points are quantum typical: the law of $\mathcal{M}_2^{\rm disk}(2)$ is invariant under independently resampling both points from the quantum length measure~\cite[Proposition A.8]{DMS-mating-of-tree}. This enables the definition of general quantum disks with quantum typical marked points. In this paper, we focus on the case of $\QD_{1,0}$ with a single interior marked point.

\begin{definition}
    Sample $(\mathcal{S},\phi,-\infty,+\infty)/{\sim_\gamma}$ from the reweighted measure $\nu_\phi(\partial \mathcal{S})^{-2} \mu_\phi(\partial \mathcal{S}) \mathcal{M}_2^{\rm disk}(2)$ and independently sample $z$ from the probability measure proportional to $\mu_\phi$. We call $(\mathcal{S},\phi,z)/{\sim_\gamma}$ a quantum disk with a single interior point and denote its law by $\QD_{1,0}$.
\end{definition}

We next introduce the quantum sphere, for which it is convenient to work on the horizontal cylinder $\mathcal{C}=\mathbb{R} \times [0,2\pi]/{\sim}$, where $(x,0) \sim (x,2\pi)$ for $x \in \mathbb{R}$ (as in Section~\ref{subsec:mc}). Let $m$ be the uniform measure on $(\{0\} \times [0,2\pi]/{\sim}$. The Gaussian free field $h_\mathcal{C}$ on $\mathcal{C}$ with normalization $\int_\mathcal{C} h_\mathcal{C}(z) m(\dd z)=0$ is constructed similar to $h_\mathcal{S}$: we define the Dirichlet inner product $\langle f, g \rangle_\nabla=(2\pi)^{-1} \int_\mathcal{C} \nabla f \cdot \nabla g$ on the space of smooth compactly-supported functions on $\mathcal{C}$ satisfying $\int_\mathcal{C} f(z)m(\dd z)=0$, and let $H(\mathcal{C})$ be its Hilbert space closure under this inner product. Then, $h_\mathcal{C}:=\sum_{n=1}^\infty \alpha_n f_n$, where $(f_n)_{n=1}^\infty$ is an orthonormal basis of $H(\mathcal{C})$ and $(\alpha_n)_{n=1}^\infty$ are i.i.d.\ standard Gaussian variables.

The space $H(\mathcal{C})$ also has an orthogonal decomposition $H(\mathcal{C})=H_1(\mathcal{C}) \oplus H_2(\mathcal{C})$, where $H_1(\mathcal{C})$ (resp.\ $H_2(\mathcal{C})$) is the subspace of functions in $H(\mathcal{C})$ which are constant (resp.\ have mean zero) on $(\{t\} \times [0,2\pi])/{\sim}$ for each $t \in \mathbb{R}$. This gives $h_{\mathcal{C}}=h_{\mathcal{C}}^1+h_{\mathcal{C}}^2$, where $h_{\mathcal{C}}^1$ and $h_{\mathcal{C}}^2$ are independent projections of $h_\mathcal{C}$ onto $H_1(\mathcal{C})$ and $H_2(\mathcal{C})$.

\begin{definition}[Quantum sphere]
    For $W>0$, we define the measure $\mathcal{M}_2^{\rm sph}(W)$ as follows. Write $\alpha=Q-\frac{W}{2\gamma}$ and let $(B_t)_{t \ge 0}$ be a standard Brownian motion conditioned on $B_t-(Q-\alpha)t<0$ for all $t>0$, and $(\widetilde B_t)_{t \ge 0}$ be its independent copy. Let
    \begin{equation*}
        Y_t=
        \begin{cases}
            B_t-(Q-\alpha) t, & \mbox{ for }\,t \ge 0, \\
            \widetilde B_{-t}+(Q-\alpha)t, & \mbox{ for }\,t<0,
        \end{cases}
    \end{equation*}
    and set $h_1(z)=Y_{\mathrm{Re} z}$ for $z \in \mathcal{C}$. Let $h_2$ be a random generalized function with the same law as $h_\mathcal{C}^2$ defined above and independent from $h_1$.
    Sample $\mathbf{c} \in \mathbb{R}$ independently from the measure $\frac{\gamma}{2} e^{(\beta-Q)c} \dd c$, and set $\phi(z)=h_1(z)+h_2(z)+\mathbf{c}$. The infinite measure $\mathcal{M}_2^{\rm sph}(W)$ is defined as the law of $(\mathcal{C},\phi,-\infty,+\infty)/{\sim_\gamma}$.
\end{definition}

For $W=4-\gamma^2$, the two marked points are quantum typical---they are independent samples from the quantum area measure~\cite[Proposition A.13]{DMS-mating-of-tree}. We therefore define $\QS_2:=\mathcal{M}_2^{\rm sph}(4-\gamma^2)$.

For a measure $\mathcal{M}$ on quantum surfaces, we can disintegrate it over the quantum lengths of its boundary arcs. For instance, we can define a disintegration $\{\QD_{1,0}(\ell)\}_{\ell>0}$ of $\QD_{1,0}$, where each $\QD_{1,0}(\ell)$ is supported on quantum disks with one interior marked point and boundary length $\ell$, satisfying $\QD_{1,0}=\int_0^\infty \QD_{1,0}(\ell) \dd \ell$.

For a measure $\mathcal{M}$ on quantum surfaces (possibly with marked points) and a conformally invariant measure $\mathcal{P}$ on curves (possibly multiple), we let $\mathcal{M} \otimes \mathcal{P}$ denote the law of the curve-decorated quantum surface obtained by sampling $(S,\eta)$ from $\mathcal{M} \times \mathcal{P}$ and drawing $\eta$ on $S$.

Conformal welding is a class of results stating that $\mathcal{M} \otimes \mathcal{P}$ arises from welding two or more \emph{independent} quantum surfaces along their boundary arcs or loops.
Specifically, let $\mathcal{M}^1$ and $\mathcal{M}^2$ be measures on quantum surfaces with boundary marked points. For $i=1,2$, fix a boundary arc $e_i$ of a sample from $\mathcal{M}^i$, and let $\{\mathcal{M}^i(\ell)\}_{\ell>0}$ be the disintegration of $\mathcal{M}^i$ over the quantum lengths of $e_i$, so that $\mathcal{M}^i=\int_0^\infty \mathcal{M}^i(\ell) \dd \ell$. The conformal welding of $\mathcal{M}^1$ and $\mathcal{M}^2$ along the boundary arcs $e_1$ and $e_2$ is defined as
\[ \int_0^\infty \mathcal{M}^1(\ell) \times \mathcal{M}^2(\ell) \dd \ell, \]
where for each $\ell>0$, $\mathcal{M}^1(\ell) \times \mathcal{M}^2(\ell)$ is the law of a quantum surface obtained by sampling a pair of independent quantum surfaces from $\mathcal{M}^1(\ell) \otimes \mathcal{M}^2(\ell)$ and then conformally welding them together according to quantum length, yielding a single quantum surface decorated with a curve (the welding interface).

For quantum surfaces $\mathcal{M}^1$ and $\mathcal{M}^2$ without boundary marked points (so that $e_1$ and $e_2$ are boundary loops), the conformal welding involves extra randomness. In this case, we sample points $p_i$ on $e_i$ from the probability measure proportional to the quantum length measure for $i \in \{1,2\}$, and then conformally weld $\mathcal{M}^1$ and $\mathcal{M}^2$ along $e_1$ and $e_2$ by identifying $p_1$ and $p_2$. Denote the law of the resulting quantum surface by $\mathcal{M}^1(\ell) \times \mathcal{M}^2(\ell)$. The uniform conformal welding of $\mathcal{M}^1$ and $\mathcal{M}^2$ along boundary loops $e_1$ and $e_2$ is defined as
\[ \int_0^\infty \mathcal{M}^1(\ell) \times \ell\,\mathcal{M}^2(\ell) \dd \ell, \]
where the factor $\ell$ accounts for the additional welding freedom.

The following proposition, taken from~\cite[Propositions 6.5]{ACSW24}, provides the key tool for characterizing $\SLE_\kappa$ loop measures through the conformal welding of $\sqrt{\kappa}$-LQG surfaces.

\begin{proposition}[\cite{ACSW24}]\label{prop:weld-sle-loop}
    For $\kappa \in (8/3,4)$ and $\gamma=\sqrt{\kappa}$, there exists a constant $C=C(\kappa)$ such that
    \begin{equation}\label{eq:weld-2-qd}
        \mathrm{QS}_2 \otimes \SLE_\kappa^{\rm sep} = C \int_0^\infty \QD_{1,0}(\ell) \times \ell\,\QD_{1,0}(\ell) \dd \ell,
    \end{equation}
    where the right-hand side represents the uniform conformal welding along the boundaries of two independent samples from $\QD_{1,0}$ conditioned to have the same quantum boundary lengths.
\end{proposition}

\subsection{Stationarity of the SLE shape measure}\label{subsec:shape}

Recall that $\mathcal{L}_\kappa(\mathcal{C})$ is the shape measure of $\SLE_\kappa^{\rm sep}(\mathcal{C})$ (Lemma~\ref{lem:shape}). We now establish the stationarity of $\mathcal{L}_\kappa(\mathcal{C})$ for the Markov chain $(\eta^i)_{i \ge 0}$, the last ingredient required to prove Theorem~\ref{thm:ratio}.

\begin{proposition}\label{prop:stationary-bcle-simple}
    Fix $\kappa \in (8/3,4)$. If $\eta^0$ is sampled from $\mathcal{L}_\kappa(\mathcal{C})$, then the law of $\eta^1$ is also $\mathcal{L}_\kappa(\mathcal{C})$.
\end{proposition}

As outlined earlier, the proof of Proposition~\ref{prop:stationary-bcle-simple} relies on SLE/LQG coupling. Specifically, we realize the loop $\mathcal{L}_{B \to R}$ (see Section~\ref{subsec:construct-interface}) as an interface arising from the conformal welding of a quantum surface with annular topology and a quantum disk. The explicit law of this quantum surface is not required; we rely solely on its symmetry property.

\begin{lemma}\label{lem:weld-blue-red}
    Let $\kappa \in (8/3,4)$ and $\rho=\frac{3\kappa}{2}-6$. Let $\mu$ denote the law of $\mathcal{L}_{B \to R}$. Then there exists a measure $\mathrm{QA}(\rho)$ on the space of quantum surfaces with annular topology such that
    \begin{equation}\label{eq:bcle-welding-simple}
        \QD_{1,0} \otimes \mu = \int_0^\infty \mathrm{QA}(\rho;\ell) \times \ell\,\QD_{1,0}(\ell) \dd \ell.
    \end{equation}
\end{lemma}
\begin{proof}
    The conformal welding identity follows from the exploration of $\mathcal{L}_{B \to R}$ together with the conformal welding results for CLE$_{\kappa'}$~\cite[Proposition 4.4]{ACSW-loop} and BCLE~\cite[Theorems 4.1 and 5.19]{LSYZ24}. Explicitly, $\mathrm{QA}(\rho)$ is constructed via the uniform conformal welding of $\widetilde{\mathrm{QA}}(W)$~\cite{LSYZ24} (for appropriate parameters $W=W(\rho)$) and $\mathrm{GA}$~\cite{ACSW-loop}.
\end{proof}

\begin{lemma}\label{lem: QA(W) welding}
    Let $\bar\mu$ be the law of a loop $\bar\eta$ that is coupled with an $\mathrm{SLE}^{\rm sep}_\kappa$ loop, for which the following welding identity holds
    $$ \mathrm{QS}_2\otimes (\mathrm{SLE}^{\rm sep}_\kappa,\bar\mu)= \iint _{\mathbb{R}_+^2}\QD_{1,0}(\ell_1)\times \ell_1\,\widetilde{\mathrm{QA}}(W;\ell_1,\ell_2)\times \ell_2\,\QD_{1,0}(\ell_2) \dd \ell_1 \dd \ell_2. $$
    Then the marginal law of $\bar\eta$ is $C_1\,\mathrm{SLE}^{\rm sep}_\kappa$ for some constant $C_1$.
\end{lemma}

\begin{proof}
    The claim follows from the symmetry $\widetilde{\mathrm{QA}}(W;a,b)=\widetilde{\mathrm{QA}}(W;b,a)$ from~\cite[Definition 3.12]{LSYZ24}, and the fact that the welding interface of two quantum disks follows the law of an SLE loop measure~\eqref{eq:weld-2-qd}.
\end{proof}

\begin{proof}[Proof of Proposition~\ref{prop:stationary-bcle-simple}]
    Recall the auxiliary loops $\eta^*$ and $\widetilde{\eta}^1$ from the construction of $\eta^1$ in Section~\ref{subsec:mc}. Since $\eta^0$ is sampled from $\mathcal{L}_\kappa(\cC)$ (i.e., the shape measure of the counting measure on CLE$_\kappa$ loops separating $\pm\infty$), by the domain Markov property of CLE$_\kappa$, 
    the law of $\eta^*$---when shifted to be an element of $\Loop_0(\mathcal{C})$---is $\mathcal{L}_\kappa(\mathcal{C})$ as well. Furthermore, conditioned on $\eta^*$, the law of $\widetilde{\eta}^1$ is the pushforward of the law of $\mathcal{L}_{B \to R}$ under the conformal map from $\mathbb{D}$ to $\mathcal{C}_{\eta^*}^+$. Since we only need to specify the shape measure of $\widetilde{\eta}^1$, we henceforth assume (with slight abuse of notation) that $\eta^*$ is sampled from $\SLE_\kappa^{\rm sep}(\mathcal{C})$. We will then show that $\widetilde{\eta}^1$ defined in this way also follows the law $\SLE_\kappa^{\rm sep}(\mathcal{C})$.
    
    Consider a quantum surface $(\mathcal{C},h,\pm \infty)/{\sim}_\gamma$ sampled from $\mathrm{QS}_2$. Independently, sample $\eta^*$ from $\SLE_\kappa^{\rm sep}(\mathcal{C})$, and then sample $\widetilde{\eta}^1$ on $\mathcal{C}_{\eta^*}^+$ according to the pushforward of $\mu$. Let $\tilde\mu$ denote the law of $\widetilde{\eta}^1$.
    Combining Proposition~\ref{prop:weld-sle-loop} and Lemma~\ref{lem:weld-blue-red}, we obtain
    \begin{equation*}
        \mathrm{QS}_2 \otimes (\SLE_\kappa^{\rm sep}, \tilde\mu) = C \iint_{\mathbb{R}_+^2} \QD_{1,0}(\ell_1) \times \ell_1\,\mathrm{QA}(\rho;\ell_1,\ell_2) \times \ell_2\,\QD_{1,0}(\ell_2) \dd \ell_1 \dd \ell_2.
    \end{equation*}
    Integrating over $\eta^*$ and applying~\eqref{eq:bcle-welding-simple} (using the same argument as Lemma~\ref{lem: QA(W) welding}), we derive
    \begin{equation}\label{eq:weld-eta-1}
        \mathrm{QS}_2 \otimes \tilde\mu = C \int_0^\infty \QD_{1,0}(\ell_2) \times \ell_2\,\QD_{1,0}(\ell_2) \dd \ell_2.
    \end{equation}
    A comparison between~\eqref{eq:weld-2-qd} and~\eqref{eq:weld-eta-1} implies that the marginal law of $\widetilde{\eta}^1$ is $\SLE_\kappa^{\rm sep}$. Indeed, we can decompose $\mathrm{QA}(\rho)$ into quantum surfaces $\widetilde{\mathrm{QA}}(W)$ and $\mathrm{GA}$, and then repeat the argument in Lemma~\ref{lem: QA(W) welding} for each single-step welding (here we also use the symmetry of $\mathrm{GA}$ under boundary interchange~\cite[Proposition 7.6]{ACSW-loop}). In particular, the law of $\eta^1$ is $\mathcal{L}_\kappa(\mathcal{C})$.
\end{proof}

\begin{proof}[Proof of Theorem~\ref{thm:ratio}]
    Combining Propositions~\ref{prop:markov} and~\ref{prop:stationary-bcle-simple}, we obtain that $\mathcal{L}_\kappa(\cC)=\wt{\mathcal{L}}_\kappa(\cC)$. Then $\SLE_\kappa^\sep=\mathsf{C}(\kappa)^2 \,\wt{\SLE}_\kappa^\sep$ follows from Lemma~\ref{lem:shape}.
\end{proof}

\section{Derivation of $\mathsf C(\kappa)$: Proof of Theorem~\ref{thm:value}}\label{sec:computation}

In this section, we compute the conformal radii of the interfaces $\mathcal L_{R \to B}$ and $\mathcal L_{B \to R}$ (Proposition~\ref{prop:cr-moment-intro}), and deduce the value of $\mathsf C(\kappa)$ in Theorem~\ref{thm:ratio}.
 
\subsection{Conformal radii of spin interfaces: Proof of Proposition~\ref{prop:cr-moment-intro}}

Recall from Section~\ref{subsec:construct-interface} the interfaces $\mathcal{L}_{R \to B}$ and $\mathcal{L}_{B \to R}$ on $\mathbb{D}$. Let $\mathsf{R}_{R \to B}=\mathrm{CR}(0,D(\mathcal{L}_{R \to B}))$ and $\mathsf{R}_{B \to R}=\mathrm{CR}(0,D(\mathcal{L}_{B \to R}))$ be the conformal radii, viewed from the origin, of the simply connected domains they enclose. This subsection is devoted to proving Proposition~\ref{prop:cr-moment-intro}, which gives the moments of $\mathsf{R}_{R \to B}$ and $\mathsf{R}_{B \to R}$.

We begin by recalling the conformal radii of BCLE from~\cite{LSYZ24}.

\begin{lemma}\label{lem:cr-bcle}
    Let $\kappa \in (2,4)$, $\kappa'=16/\kappa $, $\rho \in (-2,\kappa-4)$, and $\rho_B'=\kappa'-4+\frac{\kappa'}{4}\rho$ as in~\eqref{eq:parameter-nonsimple}.
    Denote by $\mathcal L$ the loop in $\BCLE_\kappa(\rho)$ surrounding the origin which can be either clockwise or counterclockwise, and let $D(\mathcal L)$ be the connected component of $\mathbb D \setminus \mathcal L$ that contains the origin. Let $\{ 0 \in \BCLE_\kappa^\clockwise(\rho) \}$ (resp.\ $\{ 0 \notin \BCLE_\kappa^\clockwise(\rho) \}$) the event that $\mathcal L$ is a clockwise (resp.\ counterclockwise) loop. For $\lambda> \frac{\kappa}{8}-1$ and $\theta=\frac{\pi}{\kappa} \sqrt{(4-\kappa)^2-8\kappa \lambda}$, we have
    \begin{align}
        \mathbb{E}[{\rm CR}(0,D(\mathcal L))^\lambda \1_{0 \in {\rm BCLE}^\clockwise_\kappa(\rho)}] &= \frac{\sin(\frac{\pi(4-\kappa)}{4}) \sin(\frac{2\pi}{\kappa}(\kappa-\rho-4))}{\sin(\frac{\pi(4-\kappa)}{\kappa}) \sin(\frac{\pi}{4}(\kappa - 2\rho-4))} \cdot \frac{\sin( \frac{\kappa-2\rho-4}{4} \theta)}{\sin( \frac{\kappa}{4} \theta)}, \label{eq:cr-simple-bcle-true} \\
        \mathbb{E}[{\rm CR}(0,D(\mathcal L))^\lambda \1_{0 \notin {\rm BCLE}^\clockwise_\kappa(\rho)}] &= \frac{\sin(\frac{\pi(4-\kappa)}{4}) \sin(\frac{2\pi}{\kappa}(\rho+2))}{\sin(\frac{\pi(4-\kappa)}{\kappa}) \sin(\frac{\pi}{4}(\kappa - 2\rho-4))} \cdot \frac{\sin( \frac{2\rho+8-\kappa}{4} \theta)}{\sin( \frac{\kappa}{4} \theta)}. \label{eq:cr-simple-bcle-false}
    \end{align}
    If $\lambda \le \frac{\kappa}{8}-1$, the left hand sides of~\eqref{eq:cr-simple-bcle-true} and~\eqref{eq:cr-simple-bcle-false} are infinite.
    
    Denote by $\mathcal L'$ the loop in $\BCLE_{\kappa'}(\rho_B')$ surrounding the origin which can be either clockwise or counterclockwise, and let $D(\mathcal L')$ be the connected component of $\mathbb D \setminus \mathcal L'$ that contains the origin. Let $\{ 0 \in \BCLE_{\kappa'}^\clockwise(\rho_B') \}$ (resp.\ $\{ 0 \notin \BCLE_{\kappa'}^\clockwise(\rho_B') \}$) the event $\mathcal L'$ is a clockwise (resp.\ counterclockwise) loop. For $\lambda> \frac{\kappa'}{8}-1$ and $\theta=\frac{\pi}{\kappa} \sqrt{(4-\kappa)^2-8\kappa \lambda}$, we have
    \begin{align}
        \mathbb{E}[{\rm CR}(0,D(\mathcal L'))^\lambda \1_{0 \in {\rm BCLE}^\clockwise_{\kappa'}(\rho_B')}] &= \frac{\sin(\frac{\pi(4-\kappa)}{\kappa}) \sin(-\frac{\pi}{2} \rho)}{\sin(\frac{\pi(4-\kappa)}{4}) \sin(\frac{2\pi}{\kappa}(\rho + 2))} \cdot \frac{\sin( \frac{\kappa-2\rho-4}{4} \theta)}{\sin(\theta)}, \label{eq:cr-nonsimple-bcle-true} \\
        \mathbb{E}[{\rm CR}(0,D(\mathcal L'))^\lambda \1_{0 \notin {\rm BCLE}^\clockwise_{\kappa'}(\rho_B')}] &= \frac{\sin(\frac{\pi(4-\kappa)}{\kappa}) \sin(\frac{\pi}{4}(\kappa-2\rho - 4))}{\sin(\frac{\pi(4-\kappa)}{4}) \sin(\frac{2\pi}{\kappa}(\rho + 2))} \cdot \frac{\sin( \frac{2\rho+4}{4} \theta)}{\sin(\theta)}. \label{eq:cr-nonsimple-bcle-false}
    \end{align}
    If $\lambda \le \frac{\kappa'}{8}-1$, the left hand sides of~\eqref{eq:cr-nonsimple-bcle-true} and~\eqref{eq:cr-nonsimple-bcle-false} are infinite.
\end{lemma}

\begin{proof}
    The first and second displays follow directly from~\cite[Theorem 1.6]{LSYZ24}. The third and fourth displays are derived from~\cite[Theorem 1.8]{LSYZ24} after substituting $\rho_B'=\kappa'-4+\frac{\kappa'}{4}\rho$.
\end{proof}

We now provide the proof of Proposition~\ref{prop:cr-moment-intro} based on Lemma~\ref{lem:cr-bcle}.

\begin{proof}[Proof of Proposition~\ref{prop:cr-moment-intro}]
    First, by the discussion in Section~\ref{subsec:construct-interface}, we know that $\mathcal{L}_{R \to B}$ is the unique loop in (the non-nested) $\CLE_\kappa$ that surrounds the origin. Thus,~\eqref{eq:cr-1-23-intro} follows directly from~\cite[Theorem 1]{ssw-radii}.

    We now turn to the proof of~\eqref{eq:cr-23-1-intro}. Let $\lambda > \frac{2}{\kappa'}+\frac{3\kappa'}{32}-1=\frac{\kappa}{8}+\frac{3}{2\kappa}-1$.
    Recall the construction of $\mathcal{L}_{B \to R}$ and the parameters $\rho$ and $\rho_B'$ from Section~\ref{subsec:construct-interface}. Let $D_1'$ be the domain defined in \textit{Step 1}. By the independence of each exploration step,
    \begin{equation}\label{eq:step1}
        \mathbb{E}[\CR(0,D_1')^\lambda]=\mathbb{E}[\CR(0,D(\mathcal L'))^\lambda \mathbf{1}_{0 \in \BCLE_{\kappa'}^\clockwise(\rho_B')}] \cdot C_{\rm CLE}(\lambda) + \mathbb{E}[\CR(0,D(\mathcal L'))^\lambda \mathbf{1}_{0 \notin \BCLE_{\kappa'}^\clockwise(\rho_B')}].
    \end{equation}
    Here, the domain $D(\mathcal L')$ is defined as in Lemma~\ref{lem:cr-bcle}, and $C_{\rm CLE}(\lambda)=\mathbb E[\CR(0,D_{\Gamma'})^\lambda]$, where $D_{\Gamma'}$ is the connected component of a unit disk minus a non-nested $\CLE_{\kappa'}$ process $\Gamma'$ that contains the origin. By~\cite[Theorem 1]{ssw-radii}, we have
    \begin{equation}\label{eq:cr-cle}
        C_{\rm CLE}(\lambda)=\frac{\cos(\frac{\pi(4-\kappa)}{4})}{\cos( \frac{\kappa}{4} \theta)}.
    \end{equation}
    By \textit{Step 2} and the iterative exploration rules,
    \begin{equation}\label{eq:step2.0}
    \begin{split}
        \mathbb{E}[\CR(0,D(\mathcal{L}_{B \to R}))^\lambda]&=\mathbb{E}[\CR(0,D_1')^\lambda] \times \Big( \mathbb{E}[\CR(0,D(\mathcal{L}))^\lambda \mathbf{1}_{0 \in \BCLE_\kappa^\clockwise(\rho)}] \\
        &+\mathbb{E}[\CR(0,D(\mathcal{L}))^\lambda \mathbf{1}_{0 \notin \BCLE_\kappa^\clockwise(\rho)}] \times \mathbb{E}[\CR(0,D(\mathcal{L}_{B \to R}))^\lambda] \Big),
    \end{split}
    \end{equation}
    where the domain $D(\mathcal{L})$ is defined as in Lemma~\ref{lem:cr-bcle}. Write
    \begin{align*}
        &f(\lambda):=\mathbb{E}[\CR(0,D(\mathcal{L}))^\lambda \mathbf{1}_{0 \in \BCLE^\clockwise_\kappa(\rho)}] \times \mathbb{E}[\CR(0,D_1')^\lambda], \\
        \mbox{and} \quad &g(\lambda):=\mathbb{E}[{\CR(0,D(\mathcal{L}))^\lambda \mathbf{1}_{0 \notin \BCLE^\clockwise_\kappa(\rho)}] \times \mathbb{E}[\CR(0,D_1')^\lambda]}.
    \end{align*}
    Then~\eqref{eq:step2.0} becomes $\mathbb{E}[\CR(0,D(\mathcal{L}_{B \to R}))^\lambda]=f(\lambda)+g(\lambda) \cdot \mathbb{E}[\CR(0,D(\mathcal{L}_{B \to R}))^\lambda]$. It is clear that both $f$ and $g$ are decreasing in $\lambda$; moreover, Lemma~\ref{lem:cr-bcle} implies that $f(\lambda)=\infty$ if and only if $g(\lambda)=\infty$. Therefore, $\mathbb{E}[\CR(0,D(\mathcal{L}_{B \to R}))^\lambda] < \infty$ holds if and only if $g(\lambda)<1$. Since $g(0)<1$ and $\lim_{\lambda \to (\frac{\kappa}{8}+\frac{3}{2\kappa}-1)^+} g(\lambda)=\infty$ as $C_{\rm CLE}(\lambda)$ blows up, if we set $\lambda_0:=\sup \{ \lambda \in \mathbb{R} : g(\lambda)=1 \} \in (\frac{\kappa}{8}+\frac{3}{2\kappa}-1,0)$, then $\mathbb{E}[\CR(0,D(\mathcal{L}_{B \to R}))^\lambda]=\frac{f(\lambda)}{1-g(\lambda)}$ for $\lambda>\lambda_0$, and is infinite otherwise.
    
    Let $\lambda>\lambda_0$ and denote $x=\frac{\kappa}{4} \theta$, $y=\frac{2\rho}{4} \theta$, and $z=\theta$. By combining~\eqref{eq:step1},~\eqref{eq:cr-cle}, Lemma~\ref{lem:cr-bcle}, and using elementary trigonometric identities, the numerator $f(\lambda)$ can be expressed as
    \begin{align*}
        &\quad \frac{\sin(\frac{2\pi}{\kappa}(\kappa-\rho-4))}{\sin(\frac{2\pi}{\kappa}(\rho+2))} \cdot \frac{\sin(x-y-z)}{\sin(x)} \times \left( \frac{\sin(-\frac{\pi}{2} \rho) \cos(\frac{\pi(4-\kappa)}{4})}{\sin(\frac{\pi}{4}(\kappa-2\rho - 4))} \cdot \frac{\sin(x-y-z)}{\sin(z) \cos(x)} + \frac{\sin(y+z)}{\sin(z)} \right) \\
        &=\frac{\sin(\frac{2\pi}{\kappa}(\kappa-\rho-4))}{\sin(\frac{2\pi}{\kappa}(\rho+2))} \cdot \frac{\sin(x-y-z)}{\sin(x)} \cdot \frac{\sin(\frac{\pi}{4}(\kappa-2\rho-4)) \sin(x+y+z)-\sin(\frac{\pi}{4}(\kappa+2\rho-4)) \sin(x-y-z)}{2 \cos(x) \sin(z) \sin(\frac{\pi}{4}(\kappa-2\rho-4))},
    \end{align*}
    and the denominator $1-g(\lambda)$ is given by
    \begin{align*}
        &\quad 1-\frac{\sin(y+2z-x)}{\sin(x)} \times \left( \frac{\sin(-\frac{\pi}{2} \rho) \cos(\frac{\pi(4-\kappa)}{4})}{\sin(\frac{\pi}{4}(\kappa - 2\rho-4))} \cdot \frac{\sin(x-y-z)}{\sin(z) \cos(x)} + \frac{\sin(y+z)}{\sin(z)} \right) \\
        &=\frac{\sin(x-y-z)}{\sin(x)} \cdot \frac{\sin(\frac{\pi}{4}(\kappa-2\rho-4)) \sin(x+y+2z)-\sin(\frac{\pi}{4}(-\kappa-2\rho+4)) \sin(y+2z-x)}{2 \cos(x) \sin(z) \sin(\frac{\pi}{4}(\kappa-2\rho-4))}.
    \end{align*}
    Therefore, we derive
    
    \begin{align}\label{eq:CR-blue-bc}
        \mathbb{E}[\CR(0,D(\mathcal{L}_{B \to R}))^\lambda]&=\frac{ \sin(\frac{2\pi}{\kappa}(\kappa-\rho-4))}{\sin(\frac{2\pi}{\kappa}(\rho+2))} \nonumber \\
        &\times \frac{\sin(\frac{\pi}{4}(\kappa- 2\rho-4)) \sin(\frac{\kappa + 2\rho +4}{4} \theta) - \sin(\frac{\pi}{4}(\kappa+ 2 \rho - 4)) \sin( \frac{\kappa - 2\rho - 4}{4} \theta)}{\sin(\frac{\pi}{4}(\kappa- 2\rho-4)) \sin(\frac{\kappa + 2\rho +8}{4} \theta) - \sin(\frac{\pi}{4}(\kappa+ 2 \rho - 4)) \sin( \frac{\kappa -2\rho -8}{4} \theta)}.
    \end{align}
    Finally, for $\rho=\frac{3\kappa}{2}-6$, the above expression~\eqref{eq:CR-blue-bc} simplifies to~\eqref{eq:cr-23-1-intro}. The threshold $\lambda_0$ can be obtained by solving $1-g(\lambda_0)=0$. This concludes the proof.
\end{proof}

\subsection{Derivation of the constant $\mathsf{C}(\kappa)$: Proof of Theorem~\ref{thm:value}}

Finally, we derive the exact formula for $\mathsf{C}(\kappa)$ in~\eqref{eq:c-kappa} based on Proposition~\ref{prop:cr-moment-intro}.

\begin{proof}[Proof of Theorem~\ref{thm:value}]
    Recall that $\theta=\frac{\pi}{\kappa} \sqrt{(4-\kappa)^2-8\kappa\lambda}$, we write $h(\lambda):=\frac{\dd \theta}{\dd \lambda}$ for simplicity.
    Differentiating~\eqref{eq:cr-1-23-intro} with respect to $\lambda$, we find
    \begin{equation*}
        \mathbb{E}[(\mathsf{R}_{R \to B})^\lambda \log \mathsf{R}_{R \to B}]=h(\lambda) \cdot \frac{\cos(\frac{\pi(4-\kappa)}{\kappa}) \sin(\theta)}{\cos^2(\theta)},
    \end{equation*}
    which, after substituting $\lambda=0$, yields
    \begin{equation}\label{eq:log-cr-1-23}
        \mathbb{E}[\log \mathsf{R}_{R \to B}]=h(0) \tan(\tfrac{4\pi}{\kappa}).
    \end{equation}
    Similarly, the formula for $\mathbb{E}[\log \mathsf{R}_{B \to R}]$ can be derived from~\eqref{eq:cr-23-1-intro}. Denote the numerator and denominator in~\eqref{eq:cr-23-1-intro} respectively as
    \begin{align*}
        &U(\lambda):=\sin((\kappa-2) \theta) + 2 \cos(\pi \tfrac{4-\kappa}{2}) \sin((2-\tfrac{\kappa}{2}) \theta), \\
        \mbox{and} \quad &V(\lambda):=\sin((\kappa-1) \theta) + 2 \cos(\pi \tfrac{4-\kappa}{2}) \sin((1-\tfrac{\kappa}{2}) \theta).
    \end{align*}
    Then we have
    \begin{align*}
        &U'(\lambda)=h(\lambda) \left( (\kappa-2)\cos((\kappa-2) \theta) + (4-\kappa) \cos(\pi \tfrac{4-\kappa}{2}) \cos((2-\tfrac{\kappa}{2}) \theta) \right), \\
        \mbox{and} \quad &V'(\lambda)=h(\lambda) \left( (\kappa-1)\cos((\kappa-1) \theta) + (2-\kappa) \cos(\pi \tfrac{4-\kappa}{2}) \cos((1-\tfrac{\kappa}{2}) \theta) \right).
    \end{align*}
    Denote $x=\frac{\kappa \pi}{2}$ and $y=\frac{4\pi}{\kappa}$. By substituting $\lambda=0$ and using elementary trigonometric identities,
    \begin{align*}
        U(0)&=\sin(\pi \tfrac{(4-\kappa)(\kappa-2)} {\kappa})+2\cos(\pi \tfrac{4-\kappa}{2}) \sin(\pi \tfrac{(4-\kappa)^2}{2\kappa}) \\
        &=-\sin(2x+2y)+2\cos(x)\sin(x+2y)=\sin(2y), \\
        V(0)&=\sin(\pi \tfrac{(4-\kappa)(\kappa-1)}{\kappa})+2\cos(\pi \tfrac{4-\kappa}{2}) \sin(\pi \tfrac{(4-\kappa)(2-\kappa)}{2\kappa}) \\
        &=\sin(2x+y)-2\cos(x)\sin(x+y)=-\sin(y), \\
        U'(0)&=h(0) \left( (\kappa-2)\cos(\pi \tfrac{(4-\kappa)(\kappa-2)}{\kappa}) + (4-\kappa) \cos(\pi \tfrac{4-\kappa}{2}) \cos(\pi \tfrac{(4-\kappa)^2}{2\kappa}) \right) \\
        &=h(0) ((\kappa-2)\cos(2x+2y)+(4-\kappa)\cos(x)\cos(x+2y)) \\
        &=h(0) (2\cos(x)\cos(x+2y)-(\kappa-2)\sin(x)\sin(x+2y)), \\
        V'(0)&=h(0)\left( (\kappa-1)\cos(\pi \tfrac{(4-\kappa)(\kappa-1)}{\kappa}) + (2-\kappa) \cos(\pi \tfrac{4-\kappa}{2}) \cos(\pi \tfrac{(4-\kappa)(2-\kappa)}{2\kappa}) \right) \\
        &=h(0) ((1-\kappa)\cos(2x+y)+(\kappa-2)\cos(x)\cos(x+y)) \\
        &=h(0) (-\cos(x)\cos(x+y)+(\kappa-1)\sin(x)\sin(x+y)).
    \end{align*}
    Differentiating~\eqref{eq:cr-23-1-intro} with respect to $\lambda$, and substituting $\lambda=0$, we derive
    \begin{align}\label{eq:log-cr-23-1}
        \mathbb{E}[\log \mathsf{R}_{B \to R}]&=\frac{1}{2\cos(\frac{\pi(4-\kappa)}{\kappa})} \frac{U'(0)V(0)-U(0)V'(0)}{V(0)^2} \nonumber \\
        &=\frac{1}{-2\cos(y)} \cdot h(0) \cdot \frac{2\sin^2(y)-\kappa \sin^2(x)}{\sin(y)}.
    \end{align}
    The above equation again follows from elementary trigonometric identities. Combining~\eqref{eq:c-kappa},~\eqref{eq:log-cr-1-23}, and~\eqref{eq:log-cr-23-1}, we obtain
    \[ \mathsf{C}(\kappa)^2=1+\frac{\mathbb{E}[\log \mathsf{R}_{B \to R}]}{\mathbb{E}[\log \mathsf{R}_{R \to B}]}=1+\frac{2\sin^2(y)-\kappa \sin^2(x)}{-2\sin^2(y)}=\frac{\kappa \sin^2(x)}{2 \sin^2(y)}. \qedhere \]
\end{proof}

\bibliographystyle{alpha}

\bibliography{ref}

\begin{thebibliography}{ACSW24}

\bibitem[ACSW21]{ACSW24}
Morris {Ang}, Gefei {Cai}, Xin {Sun}, and Baojun {Wu}.
\newblock {Integrability of Conformal Loop Ensemble: Imaginary DOZZ Formula and
  Beyond}.
\newblock {\em arXiv e-prints}, page arXiv:2107.01788, July 2021.

\bibitem[ACSW24]{ACSW-loop}
Morris {Ang}, Gefei {Cai}, Xin {Sun}, and Baojun {Wu}.
\newblock {SLE Loop Measure and Liouville Quantum Gravity}.
\newblock {\em arXiv e-prints}, page arXiv:2409.16547, September 2024.

\bibitem[AHS23]{AHS23}
Morris Ang, Nina Holden, and Xin Sun.
\newblock Conformal welding of quantum disks.
\newblock {\em Electron. J. Probab.}, 28:Paper No. 52, 50, 2023.

\bibitem[AHS24]{AHS21}
Morris Ang, Nina Holden, and Xin Sun.
\newblock Integrability of {SLE} via conformal welding of random surfaces.
\newblock {\em Comm. Pure Appl. Math.}, 77(5):2651--2707, 2024.

\bibitem[BDC12]{BDC12}
Vincent Beffara and Hugo Duminil-Copin.
\newblock The self-dual point of the two-dimensional random-cluster model is
  critical for {$q\geq 1$}.
\newblock {\em Probab. Theory Related Fields}, 153(3-4):511--542, 2012.

\bibitem[BDC16]{BDC-notes}
Vincent Beffara and Hugo Duminil-Copin.
\newblock Critical point and duality in planar lattice models.
\newblock In {\em Probability and statistical physics in {S}t. {P}etersburg},
  volume~91 of {\em Proc. Sympos. Pure Math.}, pages 51--98. Amer. Math. Soc.,
  Providence, RI, 2016.

\bibitem[BH19]{BH19}
St\'ephane Benoist and Cl\'ement Hongler.
\newblock The scaling limit of critical {I}sing interfaces is
  {$\mathrm{CLE}_3$}.
\newblock {\em Ann. Probab.}, 47(4):2049--2086, 2019.

\bibitem[BJ24]{baverez2024cftsleloopmeasures}
Guillaume {Baverez} and Antoine {Jego}.
\newblock {The CFT of SLE loop measures and the Kontsevich--Suhov conjecture}.
\newblock {\em arXiv e-prints}, page arXiv:2407.09080, July 2024.

\bibitem[BPZ84]{bpz-conformal-symmetry}
A.~A. Belavin, A.~M. Polyakov, and A.~B. Zamolodchikov.
\newblock Infinite conformal symmetry in two-dimensional quantum field theory.
\newblock {\em Nuclear Phys. B}, 241(2):333--380, 1984.

\bibitem[Cam24]{Camia24}
Federico Camia.
\newblock Conformal covariance of connection probabilities and fields in 2{D}
  critical percolation.
\newblock {\em Comm. Pure Appl. Math.}, 77(3):2138--2176, 2024.

\bibitem[CG25]{cg25}
Gefei {Cai} and Yifan {Gao}.
\newblock {Uniqueness of generalized conformal restriction measures and
  Malliavin-Kontsevich-Suhov measures for $c \in (0,1]$}.
\newblock {\em arXiv e-prints}, page arXiv:2502.05890, February 2025.

\bibitem[CHI15]{Ising-spin}
Dmitry Chelkak, Cl\'{e}ment Hongler, and Konstantin Izyurov.
\newblock Conformal invariance of spin correlations in the planar {I}sing
  model.
\newblock {\em Ann. of Math. (2)}, 181(3):1087--1138, 2015.

\bibitem[CHI21]{CHI-ising}
Dmitry {Chelkak}, Cl{\'e}ment {Hongler}, and Konstantin {Izyurov}.
\newblock {Correlations of primary fields in the critical Ising model}.
\newblock {\em arXiv e-prints}, page arXiv:2103.10263, March 2021.

\bibitem[CN06]{CN06}
Federico Camia and Charles~M. Newman.
\newblock Two-dimensional critical percolation: the full scaling limit.
\newblock {\em Comm. Math. Phys.}, 268(1):1--38, 2006.

\bibitem[DC20]{DC-notes}
Hugo Duminil-Copin.
\newblock Lectures on the {I}sing and {P}otts models on the hypercubic lattice.
\newblock In {\em Random graphs, phase transitions, and the {G}aussian free
  field}, volume 304 of {\em Springer Proc. Math. Stat.}, pages 35--161.
  Springer, Cham, [2020] \copyright 2020.

\bibitem[DCST17]{DCST17}
Hugo Duminil-Copin, Vladas Sidoravicius, and Vincent Tassion.
\newblock Continuity of the phase transition for planar random-cluster and
  {P}otts models with {$1 \leq q \leq 4$}.
\newblock {\em Comm. Math. Phys.}, 349(1):47--107, 2017.

\bibitem[DMS21]{DMS-mating-of-tree}
Bertrand Duplantier, Jason Miller, and Scott Sheffield.
\newblock Liouville quantum gravity as a mating of trees.
\newblock {\em Ast\'erisque}, (427):viii+257, 2021.

\bibitem[DPSV13]{Delfino_2013}
G.~Delfino, M.~Picco, R.~Santachiara, and J.~Viti.
\newblock Spin clusters and conformal field theory.
\newblock {\em J. Stat. Mech. Theory Exp.}, (11):P11011, 15, 2013.

\bibitem[DS11]{DS11}
Bertrand Duplantier and Scott Sheffield.
\newblock Liouville quantum gravity and {KPZ}.
\newblock {\em Invent. Math.}, 185(2):333--393, 2011.

\bibitem[ES88]{Edwards-Sokal}
Robert~G. Edwards and Alan~D. Sokal.
\newblock Generalization of the {F}ortuin-{K}asteleyn-{S}wendsen-{W}ang
  representation and {M}onte {C}arlo algorithm.
\newblock {\em Phys. Rev. D (3)}, 38(6):2009--2012, 1988.

\bibitem[FMS12]{CFT-textbook}
Philippe Francesco, Pierre Mathieu, and David S{\'e}n{\'e}chal.
\newblock {\em Conformal field theory}.
\newblock Springer Science \& Business Media, 2012.

\bibitem[FZ87]{Fateev:1987vh}
V.~A. Fateev and A.~B. Zamolodchikov.
\newblock {Conformal quantum field theory models in two dimensions having Z3
  symmetry}.
\newblock {\em Nucl. Phys. B}, 280:644--660, 1987.

\bibitem[GPS13]{GPS13}
Christophe Garban, G\'{a}bor Pete, and Oded Schramm.
\newblock Pivotal, cluster, and interface measures for critical planar
  percolation.
\newblock {\em J. Amer. Math. Soc.}, 26(4):939--1024, 2013.

\bibitem[Gri06]{Grimmett-rcm}
Geoffrey Grimmett.
\newblock {\em The random-cluster model}, volume 333 of {\em Grundlehren der
  mathematischen Wissenschaften [Fundamental Principles of Mathematical
  Sciences]}.
\newblock Springer-Verlag, Berlin, 2006.

\bibitem[HS13]{Ising-energy}
Cl\'{e}ment Hongler and Stanislav Smirnov.
\newblock The energy density in the planar {I}sing model.
\newblock {\em Acta Math.}, 211(2):191--225, 2013.

\bibitem[KS16]{KS16}
Antti {Kemppainen} and Stanislav {Smirnov}.
\newblock {Conformal invariance in random cluster models. II. Full scaling
  limit as a branching SLE}.
\newblock {\em arXiv e-prints}, page arXiv:1609.08527, September 2016.

\bibitem[KS19]{KS19}
Antti Kemppainen and Stanislav Smirnov.
\newblock Conformal invariance of boundary touching loops of {FK} {I}sing
  model.
\newblock {\em Comm. Math. Phys.}, 369(1):49--98, 2019.

\bibitem[KSL25]{KSL22}
Laurin K\"ohler-Schindler and Matthis Lehmkuehler.
\newblock The fuzzy {P}otts model in the plane: scaling limits and arm
  exponents.
\newblock {\em Probab. Theory Related Fields}, 191(1-2):287--359, 2025.

\bibitem[KW16]{KW16}
Antti Kemppainen and Wendelin Werner.
\newblock The nested simple conformal loop ensembles in the {R}iemann sphere.
\newblock {\em Probab. Theory Related Fields}, 165(3-4):835--866, 2016.

\bibitem[LGJ21]{Lafay-web}
Augustin Lafay, Azat~M Gainutdinov, and Jesper~Lykke Jacobsen.
\newblock ${U}_{q}\left({\mathfrak{s}\mathfrak{l}}_{n}\right)$ web models and
  $\mathbb{Z}_{n}$ spin interfaces.
\newblock {\em Journal of Statistical Mechanics: Theory and Experiment},
  2021(5):053104, 2021.

\bibitem[LSYZ24]{LSYZ24}
Haoyu {Liu}, Xin {Sun}, Pu~{Yu}, and Zijie {Zhuang}.
\newblock {The bulk one-arm exponent for the CLE$_{\kappa'}$ percolations}.
\newblock {\em arXiv e-prints}, page arXiv:2410.12724, October 2024.

\bibitem[{Man}25]{Manolescu-review}
Ioan {Manolescu}.
\newblock {Exploring the phase transition of planar FK-percolation}.
\newblock {\em arXiv e-prints}, page arXiv:2502.08394, February 2025.

\bibitem[MS24]{miller2024existence}
Jason Miller and Lukas Schoug.
\newblock Existence and uniqueness of the conformally covariant volume measure
  on conformal loop ensembles.
\newblock In {\em Annales de l'Institut Henri Poincare (B) Probabilites et
  statistiques}, volume~60, pages 2267--2296. Institut Henri Poincar{\'e},
  2024.

\bibitem[MSW17]{MSW2017}
Jason Miller, Scott Sheffield, and Wendelin Werner.
\newblock C{LE} percolations.
\newblock {\em Forum Math. Pi}, 5:e4, 102, 2017.

\bibitem[MSW21]{MSW21-nonsimple}
Jason Miller, Scott Sheffield, and Wendelin Werner.
\newblock Non-simple conformal loop ensembles on {L}iouville quantum gravity
  and the law of {CLE} percolation interfaces.
\newblock {\em Probab. Theory Related Fields}, 181(1-3):669--710, 2021.

\bibitem[MSW22]{MSW22-simple}
Jason Miller, Scott Sheffield, and Wendelin Werner.
\newblock Simple conformal loop ensembles on {L}iouville quantum gravity.
\newblock {\em Ann. Probab.}, 50(3):905--949, 2022.

\bibitem[NRJ24]{NRJ24}
Rongvoram Nivesvivat, Sylvain Ribault, and Jesper~Lykke Jacobsen.
\newblock Critical loop models are exactly solvable.
\newblock {\em SciPost Phys.}, 17(2):Paper No. 029, 66, 2024.

\bibitem[NW11]{NacuWerner-carpet}
\c{S}erban Nacu and Wendelin Werner.
\newblock Random soups, carpets and fractal dimensions.
\newblock {\em J. Lond. Math. Soc. (2)}, 83(3):789--809, 2011.

\bibitem[She09]{shef-cle}
Scott Sheffield.
\newblock Exploration trees and conformal loop ensembles.
\newblock {\em Duke Math. J.}, 147(1):79--129, 2009.

\bibitem[Smi01]{smirnov-cardy}
Stanislav Smirnov.
\newblock Critical percolation in the plane: conformal invariance, {C}ardy's
  formula, scaling limits.
\newblock {\em C. R. Acad. Sci. Paris S\'{e}r. I Math.}, 333(3):239--244, 2001.

\bibitem[Smi10]{Smi10}
Stanislav Smirnov.
\newblock Conformal invariance in random cluster models. {I}. {H}olomorphic
  fermions in the {I}sing model.
\newblock {\em Ann. of Math. (2)}, 172(2):1435--1467, 2010.

\bibitem[SSW09]{ssw-radii}
Oded Schramm, Scott Sheffield, and David~B. Wilson.
\newblock Conformal radii for conformal loop ensembles.
\newblock {\em Comm. Math. Phys.}, 288(1):43--53, 2009.

\bibitem[SW05]{SW05}
Oded Schramm and David~B. Wilson.
\newblock S{LE} coordinate changes.
\newblock {\em New York J. Math.}, 11:659--669, 2005.

\bibitem[SW12]{shef-werner-cle}
Scott Sheffield and Wendelin Werner.
\newblock Conformal loop ensembles: the {M}arkovian characterization and the
  loop-soup construction.
\newblock {\em Ann. of Math. (2)}, 176(3):1827--1917, 2012.

\bibitem[SW16]{SW16}
Scott {Sheffield} and Menglu {Wang}.
\newblock {Field-measure correspondence in Liouville quantum gravity almost
  surely commutes with all conformal maps simultaneously}.
\newblock {\em arXiv e-prints}, page arXiv:1605.06171, May 2016.

\bibitem[Zam05]{zamolodchikov-gmm}
Al.~B. Zamolodchikov.
\newblock A three-point function of minimal {L}iouville gravitation.
\newblock {\em Teoret. Mat. Fiz.}, 142(2):218--234, 2005.

\bibitem[Zha21]{zhan-loop-measures}
Dapeng Zhan.
\newblock S{LE} loop measures.
\newblock {\em Probab. Theory Related Fields}, 179(1-2):345--406, 2021.

\end{thebibliography}

\end{document}